\documentclass[twoside,11pt]{article}
\usepackage{amsmath}
\usepackage{amsfonts}\usepackage{amsthm}
\usepackage{graphicx,pgf}
\usepackage{tikz} 
\usepackage{comment}
\usepackage{bm}
\usepackage{natbib,enumerate}

\newtheorem{thm}{Theorem}[section]
\newtheorem{cor}[thm]{Corollary}

\newtheorem{lem}[thm]{Lemma}

\newtheorem{remark}[thm]{Remark}

\newcommand{\R}{{\mathbb R}}       
\newcommand{\K}{\mathcal{K}}
\newcommand{\N}{\mathcal{N}}
\newcommand{\ra}{\rightarrow}    
  
\newcommand{\A}{{\mathcal A}} 
   
\renewcommand{\L}{{\mathcal L}}

\newcommand{\integral}{\int_{0}^{\infty}}
\newcommand{\wn}{w^{(n)}}
\newcommand{\un}{u^{(n)}}
\newcommand{\unew}{u_{new}}
\newcommand{\vnew}{v_{new}}
\newcommand{\ep}{\epsilon}

\newcommand{\ltwo}{\mathbf{l}^{T}_{2}}
\newcommand{\vfg}{v_{f-g}}

\title{Existence of standing pulse solutions to a skew-gradient system}
\author{Yung-Sze Choi\thanks{Department of Mathematics, University of Connecticut, Storrs} \and Jieun Lee\thanks{Department of Mathematics, University of Connecticut, Storrs}}

\begin{document}
\maketitle
\begin{abstract}
Reaction-diffusion systems have been primary tools for studying pattern formation. A skew-gradient system is well known to encompass a class of activator-inhibitor type reaction-diffusion systems that exhibit localized patterns such as fronts and pulses. While there is a substantial literature for the case of a linear inhibitor equation, the study of nonlinear inhibitor effect is still limited. To fill this research gap, we investigate standing pulse solutions to a skew-gradient system in which both activator and inhibitor reaction terms inherit nonlinear structures. Using a variational approach that involves several nonlocal terms, we establish the existence of standing pulse solutions with a sign change. In addition, we explore some qualitative properties of the standing pulse solutions. 
\end{abstract}

 \date{}

\section{Introduction}
\setcounter{equation}{0}
\renewcommand{\theequation}{\thesection.\arabic{equation}}

From vegetation patterns in an ecological system to propagating waves in a nerve fiber, fascinating patterns emerge in nature. 
These self-organizing structures, free of external input, may originate from homogeneous media through some spatial modulation
due to diffusion-driven Turing instability. Other patterns can represent phenomena far away from an equilibrium state; both standing and traveling waves are examples of the latter kind. In fact a standing or traveling front connects distinct equilibria, while  a pulse returns to the same steady state after undergoing
a large amplitude excursion. A pulse resembles a localized sharp spike and results from a delicate balance between gain and loss in the governing 
reaction kinetics. 
Competing mechanism, like in
activator-inhibitor systems such as FitzHugh-Nagumo and Gierer-Meinhardt equations, are therefore prime examples for pattern formation. 
Under appropriate circumstances dynamics of these pulses and their mutual interaction can be particle-like, and are referred to as dissipative solitons
\citep{Akhmediev:2008, Liehr:2013}. 
They are the building blocks for more complex structures.

In this paper, we study the existence of standing pulse solutions for a system of reaction-diffusion equations of the form
\begin{equation} \label{FN_nonlinear}
	\begin{cases}  
		u_t = du_{xx} + f(u) - v, \\ 
		\tau v_t = v_{xx}-\gamma v -v^3 + u,
	\end{cases}
\end{equation}
{on the infinite domain $(-\infty,\infty)$,}
where $f(u) = u(1-u)(u-\beta)$ and $d$, $\tau, \gamma$ and $\beta$ are positive constants.  It is a skew-gradient system which involves
an activator $u$ and an inhibitor $v$.

Restricted ourselves to two species cases, we consider  the reaction-diffusion system
\begin{equation} \label{skew-grad}
	\begin{cases}
		\tau_1 u_t = d_1 \Delta u + H_u(u, v),  \\
		\tau_2 v_t = d_2 \Delta v + (-1)^{k}H_v(u, v),
	\end{cases}
\end{equation}
where $\tau_ i > 0$ and $d_i > 0$ for $i=1,2$, $k \in \{0, 1\}$ and $H: \R^2 \ra \R$ is some smooth function. The system in \eqref{skew-grad} is said to have a skew-gradient structure if $k = 1$  \citep{Yanagida:2002b, Yanagida:2002a} and a gradient structure if $k=0$.
For a reaction-diffusion system with a gradient structure, 
there is a Lyapunov functional which eases the analysis of the time dependent problems; it also serves as a natural
variational functional for studying the stationary problem. The corresponding analysis 
of a skew-gradient system is more delicate.

A well-studied skew-gradient model that generates standing pulse solutions is
\begin{equation} \label{FN_parabolic}
	\begin{cases}  
		u_t = du_{xx} + f(u) - v,  \\ 
		\tau v_t = v_{xx} + u -\gamma v,
	\end{cases}
\end{equation}
which is referred to as the FitzHugh-Nagumo equations (the original FitzHugh-Nagumo model does not have the term $v_{xx}$, see \citealt{Fitzhugh:1961, Nagumo:1962}).
For finite domains a variational formulation of the above problem readily yields a global minimizer that corresponds to a steady state solution.
However such solution is usually oscillatory and is not a single localized sharp spike when the domain is large.  In fact when the domain is unbounded,  there is no global minimizer and a more careful treatment is necessary.
In \cite{Klaasen:1984}, the existence of positive standing pulse solutions to \eqref{FN_parabolic} was established for large $\gamma$ and large $d$ by a shooting argument when the parameters allow the presence of multiple constant steady states.
Using a special transformation to convert the equations to a quasi-monotone system for large $\gamma$ and $d$, 
 \citet{Reinecke:1999} employed comparison functions and finite domain approximation in $\R^N$ to establish a positive radially 
symmetric standing pulse solution. In \citet{Chen:2012}, a variational approach was applied to find solutions with a sign change when the activator diffusivity is small compared to that of inhibitor, i.e. $d \ll 1$. The solution obtained is a local, rather than global, minimizer.
There are also numerous numerical works on this model. Typically they are continuation type methods which require  good initial guesses to
start the algorithm. Recently
 a robust steepest descent algorithm for finding the waves numerically without a good initial guess has  been proposed in
 \citet{choi:2019a}.

When $f(u)$ is replaced by 
$f(u)/d$ in \eqref{FN_parabolic} for small $d$, this corresponds to studying the equations in a different parameter regime.
One can employ other well established methods, for example $\Gamma$-convergence or the geometric perturbation method,  to
 study standing pulses and their corresponding stability.  Some related models like Ohta-Kawasaki  involve a volumetric constraint.
See for example  \cite*{chen:2018a}; \cite{chen:2018b, vS, WW, Ren:2008} and the many references therein.

Over the past two decades, the study of \eqref{FN_parabolic} has further stretched into various extensions of the equations. 
An extension to a three-component system of \eqref{FN_parabolic} with an additional inhibitor equation of linear form has also been considered in \citet{Bode:2002, Doelman:2009, vanHeijster:2019} and the references therein. The existence of the corresponding standing and traveling pulse solutions has been investigated both  analytically and numerically. 
Despite the volume of the work on the system \eqref{FN_parabolic}, the study has been limited to the case of linear dependence of inhibitor reaction term. Although \citet{Chen:2009} considered a nonlinear inhibitor equation by adding $h(v) \in C^1$ that satisfies $h(0) = h'(0) = 0$ and $v \,h(v) \geq 0$, their existence result requires the domain to be bounded and follows from a minimax theorem established by Benci and Rabionowitz \citep{br:1979}. {Such saddle point type solution is unstable and our interest lies in local or global minimizer.}
Taking the special case $h(v)=v^3$,
in this paper we study the existence of standing pulse solutions on the real line in the presence of cubic nonlinearity in the inhibitor equation of \eqref{FN_parabolic}. 
Specifically, we study the steady-state of \eqref{FN_nonlinear}, namely the system 
\begin{equation} \label{FN_nonlinear_steady}
	\begin{cases}  
		du_{xx} + f(u) - v = 0,  \\ 
		v_{xx}-\gamma v -v^3 + u = 0,
	\end{cases}
\end{equation}
on $(-\infty, \infty)$ for small $\gamma$ and $d$. Observe that the system \eqref{FN_nonlinear_steady} has a skew-gradient structure with 
\[
H(u, v) =  \frac{1}{2} \gamma v^2 + \frac{1}{4}v^4 - uv - F(u),
\] where $F(u) = - \int_{0}^{u} f(x) \, dx = u^4/4-(1+\beta)u^3/3 +  \beta u^2/2$. {To our best knowledge, this work is the first attempt to show the existence of standing pulse solutions on $(-\infty, \infty)$ in the skew-gradient system \eqref{FN_parabolic} that accounts for the nonlinear dependence of inhibitor reaction term. The use of the explicit form of the Green's function in the case of linear inhibitor equation needs to be substantially modified. The additional nonlinearity may enable the model to capture more complex behavior of standing pulse solutions. Other kinds of 
nonlinearity  associated with more general skew-gradient systems can be studied later on as the techniques we develop in this work may apply to a broader class of skew-gradient systems.} 
We will look for solutions $(u, v)$ that are even in $x$ and 
\[
\lim_{|x| \ra \infty} (u, v) = (0, 0).
\]
Due to the symmetry, we restrict our attention to $[0, \infty)$. The anchor at the origin prevents the solution from translation,
which is important in analyzing the equations.
Our main result is summarized in the following theorem:

\begin{figure}[t]
\begin{center}
\includegraphics[trim=20 50 20 150, width=.6\textwidth]{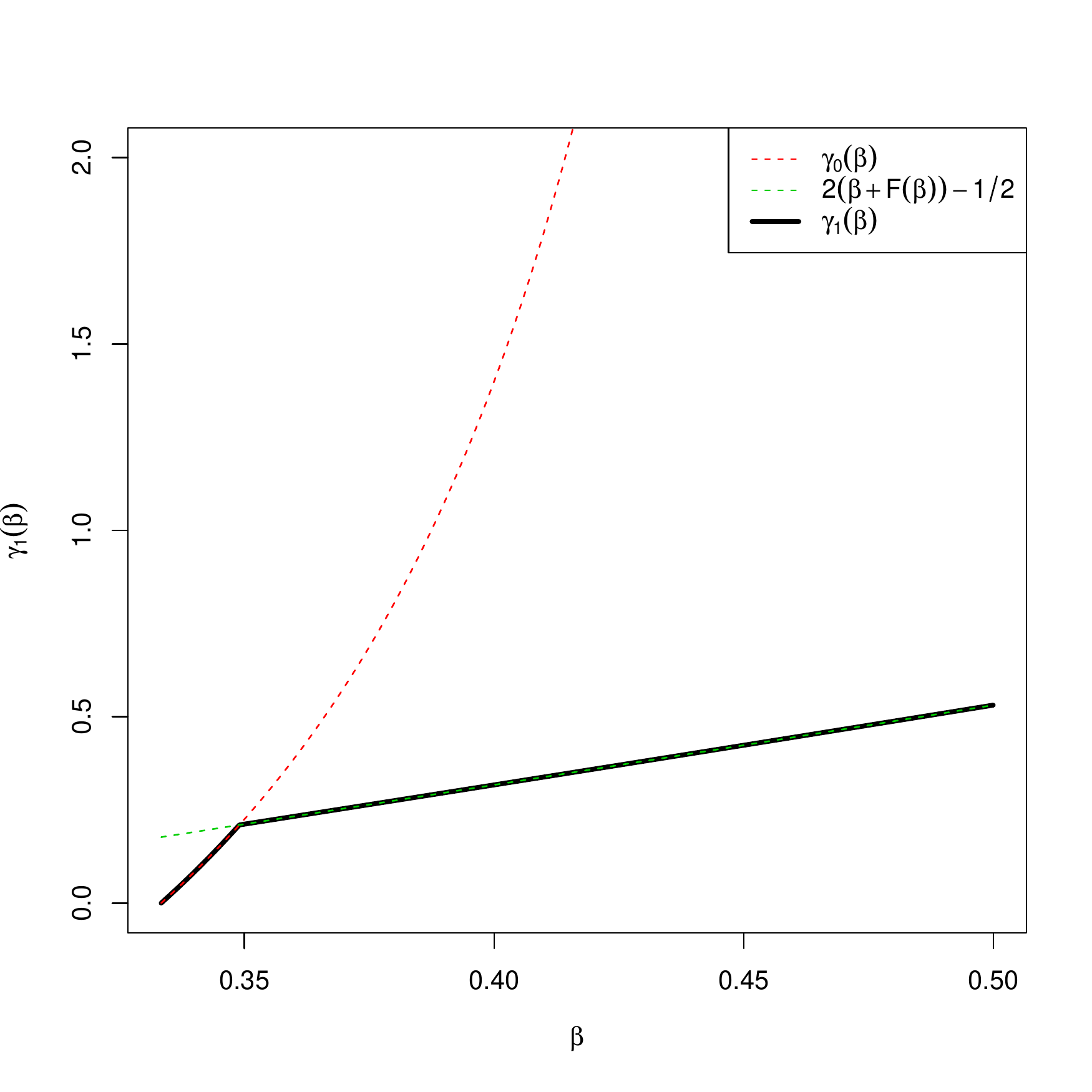}
\end{center}
\caption{A plot of $\gamma_1$ versus $\beta$ when $\beta \in (1/3,1/2)$. Here
 $\gamma_0= 3 \beta^2/(1-2\beta) -1$ and $\gamma_1=\min \{ \gamma_0, 2(\beta+F(\beta))-1/2 \}$.
 For $\gamma<\gamma_1$, there is a standing pulse solution when $d$ is sufficiently small.}
 \label{fig:1}
\end{figure}

\begin{thm} \label{theorem1}
Let $\beta \in (1/3, 1/2)$ be given. There exists a $\gamma_1 > 0$ so that for any $\gamma \in (0,\gamma_1]$, we have a
$d_1 = d_1(\gamma) > 0$ such that whenever $\gamma < \gamma_1$ and $d < d_1$, then \eqref{FN_nonlinear_steady} has a solution which is denoted by $(u_0, v_0)$ with $u_0, v_0 \in C^{\infty}(0, \infty)$ and  exponentially decay to 0 as $x \ra \infty$; that is, \eqref{FN_nonlinear_steady} possesses a standing pulse solution. 
\end{thm}
\noindent 
We have an explicit estimate  for $\gamma_1$ in Lemma~\ref{lem_gamma1}. 
To get a sense of the constraint on $\gamma$ in the above theorem, a plot of $\gamma_1$ for $\beta \in (1/3,1/2)$ is presented in Figure~\ref{fig:1}. For $\gamma< \gamma_1$ a standing pulse solution exists if $d \leq d_1(\gamma)$. To simplify notation, we suppress the dependence of $\beta$ when we refer to $\gamma_1$ or $d_1$; for instance, we write $d_1(\gamma)$ rather than $d_1(\beta, \gamma)$. Some qualitative properties of the above standing pulse solution is established in the next theorem.

\begin{thm}\label{theorem2}
Suppose that $(u_0, v_0)$ is a standing pulse solution obtained by Theorem \ref{theorem1}. Then
\begin{enumerate}[(i)] \itemsep=0.25mm
\item There is a pair of unique points $0<x_1<x_2< \infty$ such that $u_0(x_1) = \beta$ and $u_0(x_2) = 0$, respectively. Moreover $u'_0(x_1) < 0$ and $u'_0(x_2) < 0$.
\item $u_0 >\beta$ on $[0,x_1)$, $0<u_0<\beta$ on $(x_1,x_2)$ and $u_0<0$ on $(x_2,\infty)$.
\item $u_0'<0$ on $[x_1,x_2]$.
\item $u_0$ possesses one global negative minimum on $(x_2, \infty)$; this is also the unique local minimum point of $u_0$ on $[x_1,\infty)$.
\item $v_0 > 0$ on $[0, \infty)$. 
\end{enumerate}
\end{thm}

The remainder of this paper is organized as follows. In section \ref{sec_prelim}, we show the existence of a nonlinear operator $\N$ such that for any given $u \in H^1(0, \infty)$, $v = \N u \in H^3(0, \infty)$ solves (\ref{FN_nonlinear_steady}b) uniquely. We further investigate the properties of $\N$ including its (Fr\'echet) differentiability. Section \ref{sec_varForm} introduces a functional $\hat{J}$ whose minimizer corresponds to the solution of \eqref{FN_nonlinear_steady}. It can be concluded from the analysis in \cite{Chen:2012} that a global minimizer of $\hat{J}$ does not exist. We therefore introduce a class of admissible functions $\A$ and consider a minimizer of $J = \hat{J}|_{\A}$. We note that with the nonlinear reactions in both equations of \eqref{FN_nonlinear_steady}, $J$ involves two nonlocal terms. A substantial part of Section \ref{sec_varForm} is dedicated to the treatment of the nonlocal terms. The positivity of the nonlocal term as well as the bounds of $\N u$ are discussed. Section \ref{sec_seqEstimate} derives a priori estimates for a minimizing sequence of $J$, and a minimizer $u_0 \in \A$ with $J(u_0) < 0$ is extracted from the minimizing sequence in Section \ref{sec_existence}. Our main task in subsequent sections is to show that the constraints imposed on $\A$ are not actively engaged. In sections \ref{sec_corner} and \ref{sec_positive}, some essential properties of the minimizer $u_0$ and $v_0 = \N u_0$ are established. We show that $u_0 \in C^1$ which in turn allows positivity of $v_0 = \N u_0$ to be shown. Such properties then enable us to eliminate the possibility of $u_0$ equals to one of the constraints in Sections \ref{sec_constraint0} and \ref{sec_constraintBeta}. By showing that the constraints imposed on $\A$ are in fact inactive, we conclude that the {minimizer $u_0$} is a standing pulse solution of \eqref{FN_nonlinear_steady}. 

\section{A nonlinear inhibitor equation} \label{sec_prelim}
\setcounter{equation}{0}
\renewcommand{\theequation}{\thesection.\arabic{equation}}

When a variational method is employed to find a standing pulse solution of \eqref{FN_parabolic}, one introduces a linear operator $\cal L$ associated with the inhibitor equation so that $v= {\cal L} u$. This section serves as a counterpart when we are confronted
with a nonlinear inhibitor equation.
For any given $u \in H^1(0, \infty)$, we show that there exists a nonlinear operator $\N: H^1(0, \infty) \ra H^3(0,\infty)$ such that $v= \N u$ satisfies (\ref{FN_nonlinear_steady}b). 
It is also necessary to examine the (Fr\'echet) differentiability of this operator ${\cal N}$ in our new variational formulation.
While properties for the linear operator ${\cal L}$ is more or less obvious, the same cannot be said about ${\cal N}$.
We begin with some basic estimates.

\begin{lem} \label{u_embedding}
Suppose $u \in H^1(0, \infty)$ and $u_1, u_2 \in H^1(0, \infty)$. Then 
\begin{enumerate}[(i)] \itemsep=0.25mm
\item $\|u\|_{L^{\infty}(0, \infty)} \leq \sqrt{2}\,\|u\|_{H^1(0, \infty)}$ and $u(x) \ra 0$ as $x \ra 0$. 
{\item $\|u_1u_2\|_{H^1(0, \infty)} \leq \sqrt{5}\,\|u_1\|_{H^1(0,\infty)}\|u_2\|_{H^1(0,\infty)}$.}
\end{enumerate}
\end{lem}

\begin{proof}
Given any $a \in [0, \infty)$, let $a \leq t < x \leq a+1$. By integrating both sides of
\begin{equation*}
u^2(x) = u^2(t) + \int_{t}^{x} D(u^2(s)) \, ds
\end{equation*}
with respect to $t$ over the interval $(a, a+1)$, we obtain from the Young's inequality
\begin{align*}
\|u\|^2_{L^{\infty}(a, a+1)} &\leq \|u\|^2_{L^2(a, a+1)} + 2\|u\|_{L^2(a, a+1)}\|Du\|_{L^2(a, a+1)} \\
&\leq 2\|u\|^2_{H^1(a, a+1)} \;.
\end{align*}
Taking the supremum over $a \in [0, \infty)$ yields $\|u\|_{L^{\infty}(0,\infty)} \leq \sqrt{2} \, \|u\|_{H^1(0,\infty)}$. Since $\|u\|_{H^1(a, \infty)} \ra 0$ as $a \ra \infty$, it is clear that $u \ra 0$ as $x \ra \infty$. This completes the proof of (i). Next, statement (ii) follows from
\begin{align*}
\|u_1u_2\|_{H^1} &= \left( \|u_1u_2\|^2_{L^2} + \| u_2 \, Du_1+ u_1\, Du_2 \|_{L^2}^2 \right)^\frac{1}{2} \\
&\leq \left( \|u_1\|^2_{L^2}\|u_2\|^2_{L^2} + 2 \| Du_1 \|^2_{L^2} \| u_2\|^2_{L^2} + 2  \|u_1\|^2_{L^2}\| D u_2\|_{L^2}^2 \right)^\frac{1}{2} \\
&\leq (5 \|u_1\|^2_{H^1} \|u_2\|^2_{H^1})^\frac{1}{2}\,.
\end{align*}
\end{proof}

\begin{lem} \label{v_embedding}
{Assume  $\gamma > 0$, $f \in H^1(0, \infty)$ and $p \in H^1(0, \infty)$ with $p \geq 0$. If $v \in H^1(0, \infty)$ satisfies
\begin{equation} \label{v_and_f}
\int_0^\infty \left\{v_x \varphi_x + (\gamma + p)v \varphi \right\} \, dx = \int_0^\infty f \varphi \, dx \;,  \quad \text{$\forall \varphi \in H^1(0, \infty)$}
\end{equation}
then $v \in H^3(0, \infty)$ and $\|v\|_{H^3(0, \infty)} \leq C_0\|f\|_{H^1(0, \infty)}$ for some positive constant $C_0 = C_0(\gamma, \|p\|_{H^1})$.}
\end{lem}

\begin{proof}
As $p \in L^{\infty}(0,\infty)$ by Lemma~\ref{u_embedding}, we have $p v \in L^2(0,\infty)$. {By choosing $\varphi = v$ in \eqref{v_and_f},}
it is immediate from regularity estimate that  $v \in H^2(0,\infty)$ and satisfies
$v_{xx} = -f + (\gamma + p) v$ a.e. {Moreover, we see that $\|v\|_{H^1(0, \infty)} \leq \max\{1, 1/\gamma\}\|f\|_{L^2(0, \infty)}.$} Finally we observe 
\begin{align*}
\|v_{xx}\|_{H^1} &\leq \gamma\|v\|_{H^1} + \sqrt{5} \,\|p\|_{H^1}\|v\|_{H^1}  + \|f\|_{H^1} \\
&\leq \left( \gamma + \sqrt{5} \,\|p\|_{H^1} \right)\max\{1, 1/\gamma\} \|f\|_{H^1} + \|f\|_{H^1} \;.
\end{align*}
Therefore, $\|v\|_{H^3} \leq C_0\|f\|_{H^1}$ for some positive constant $C_0 = C_0(\gamma, \|p\|_{H^1})$.
\end{proof}

\begin{lem} \label{v_variation}
Given $u \in H^1(0, \infty)$, define a functional $\K:H^1(0, \infty) \ra \R$ such that whenever $z \in H^1(0,\infty)$
\begin{equation*}
 \K(z) \equiv \integral \left\{\frac{z_x^2}{2} + \frac{\gamma z^2}{2} + \frac{z^4}{4} - uz \right\} \, dx. \label{var_K}
 \end{equation*}
\noindent Then the followings hold: \vspace{3mm} \\
\noindent (i) $\K$ is well defined.  \\
(ii) $\K$ is Fr\'echet differentiable with
\begin{equation*}
\K'(z)\varphi = \integral \left \{ z_x\varphi_x + \gamma z \varphi + z^3\varphi - u\varphi \right \} \, dx,  \quad \text{$\forall \varphi \in H^1(0, \infty)$}.
\end{equation*}
(iii) $\K$ has a minimizer $v \in H^1(0, \infty)$ which is a weak solution of (\ref{FN_nonlinear_steady}b), i.e.
\begin{equation*}
\integral \left \{ v_x\varphi_x + \gamma v \varphi + v^3\varphi - u\varphi \right \} \, dx = 0, \quad \text{$\forall \varphi \in H^1(0, \infty)$}. \\ 
\end{equation*}
\hspace{\parindent} Moreover, $v \in H^3(0,\infty)$ and satisfies $v_{xx} - \gamma v  -v^3 +u=0$  a.e. \\
\noindent (iv) The weak solution $v$ is unique. 
\end{lem}

\begin{proof}
For any $z \in H^1(0, \infty)$, it follows from Lemma~\ref{u_embedding} that $\|z\|_{L^{\infty}(0, \infty)} \leq \sqrt{2}\,\|z\|_{H^1(0, \infty)}$. Therefore
\begin{equation*}
 \lvert \K(z) \rvert  \leq  \frac{1}{2} \max\{1, \gamma\} \|z\|^2_{H^1} + \frac{1}{4}\|z\|^2_{L^\infty} \|z\|^2_{L^2} + \|u\|_{L^2}\,\|z\|_{L^2} < \infty. 
\end{equation*}
This completes the proof of (i). Statement (ii) is standard. As a consequence of convexity and coercivity of $\K$, a minimizer $v \in H^1(0, \infty)$ exists and $\K^{\prime}(v)\varphi = 0$ for all $\varphi \in H^1(0, \infty)$. Observe that with $p:= v^2 \in H^1(0,\infty)$, it follows from 
Lemma~\ref{v_embedding} that 
$v \in H^3$.
These prove (iii).
Next, suppose $v_1$ and $v_2$ are weak solutions of (\ref{FN_nonlinear_steady}b) with $v_1 \neq v_2$. Then
\begin{equation*}
\integral \{ (v_1 - v_2)_x\varphi_x + \gamma (v_1 - v_2)\varphi + (v_1^3 - v_2^3)\varphi \} \, dx = 0, \quad \text{$\forall \varphi \in H^1(0, \infty)$}.
\end{equation*}
By choosing $\varphi= v_1 - v_2,$ we have
\[ \integral \{ (v_1 - v_2)^2_x + \gamma (v_1 - v_2)^2 + (v_1^2 + v_1v_2 + v_2^2)(v_1 -v_2)^2 \} \, dx = 0. \]
From $(v_1^2 + v_1v_2 + v_2^2) \geq 0$, it is clear that $v_1 - v_2 = 0$ as desired in (iv).
\end{proof}

\begin{lem} \label{natural_bndry}
If $v$ is a critical point of $\K$ defined in Lemma~\ref{v_variation}, then $v_{x}(0) = 0$.
\end{lem}

\begin{proof}
A critical point $v \in H^2(0, \infty)$ of $\K$ satisfies
\begin{align*}
0 = \K'(v)\varphi &= \integral \left \{ v_x\varphi_x + \gamma v \varphi + v^3\varphi - u\varphi \right \} \, dx \\ 
&= \integral \{ -v_{xx} + \gamma v + v^3 - u \} \varphi \, dx \, - \, v_{x}(0)\varphi(0)
\end{align*}
for all compactly supported $\varphi \in C^{\infty}[0, \infty)$. Since we know from Lemma~\ref{v_variation} that $v$ satisfies (\ref{FN_nonlinear_steady}b) a.e., we have $v_{x}(0)\varphi(0) = 0$
for any arbitrary $\varphi(0)$. Therefore, $v_{x}(0) = 0$.
\end{proof}

\begin{remark} 
The property in Lemma~\ref{natural_bndry} is well known and often referred to as a natural boundary condition. 
\end{remark}

Suppose $u \in H^1(0,\infty)$ and let $v \in H^3(0,\infty)$ be the unique minimizer of $\K$ in Lemma~\ref{v_variation}. Then
we write  $v := \N u$ so that $\N : H^1(0,\infty) \to H^3(0,\infty)$ and 
$v_x(0)=0$. We remark that $u \in C^{1/2}[0, \infty)$ and $v \in C^{2+1/2}[0, \infty)$ by the Sobolev embedding and therefore $(u_0, v_0)$ satisfies (\ref{FN_nonlinear_steady}b) in a classical sense. Finding a symmetric solution to the system $\eqref{FN_nonlinear_steady}$ becomes equivalent to 
studying the integral-differential equation
\begin{equation*} 
du_{xx} + f(u) - \N u = 0
\end{equation*}
with boundary condition $u_x(0)=0$.
Before closing this section, we present some properties of the nonlinear operator $\N$ that will be used throughout this paper.

\begin{lem} \label{w_and_Nw}
For any $w \in H^1$, 
\begin{equation}
\|\N w \|_{H^1(0, \infty)} \leq \max\{1, 1/\gamma\}\|w\|_{L^2(0, \infty)}.
\end{equation}
\end{lem}

\begin{proof}
Multiplying (\ref{FN_nonlinear_steady}b) through by $\N w$ and integrating by parts,
\[ \integral \{(\N w)^{\prime \, 2} + \gamma (\N w)^2 + (\N w)^4 \} \, dx = \integral w\N w \, dx \]
and the result follows.
\end{proof}

The next lemma shows that $\N$ is Frech\'et differentiable. Its derivative will be denoted by $\N'$.
\begin{lem} 
The nonlinear map ${\N}$ is Frech\'et differentiable. To be precise,
given any $w \in H^1(0, \infty)$ and $v = \N w$, we have $\N^{\prime}(w): H^1(0, \infty) \ra H^3(0, \infty)$ such that for any given $\hat{w} \in H^1(0, \infty)$
\begin{equation*}
\hat{v} = \N^{\prime}(w)\hat{w}
\end{equation*}
is the unique solution in $H^3(0, \infty)$ of
\begin{equation} \label{v_hat1}
\hat{v}^{\prime\prime} - \gamma \hat{v} - 3v^2\hat{v} = -\hat{w}
\end{equation}
with $v^{\prime}(0) =0$.
\end{lem}

\begin{proof}
Fix $w \in H^1(0, \infty)$ and set $v = \N w$. Given $\hat{w} \in H^1(0, \infty)$, let $A: H^1(0, \infty) \ra H^3(0, \infty)$ be a map such that $\hat{v} = A\hat{w}$ is 
the unique $H^3$ solution of \eqref{v_hat1}. The existence of $A$ is guaranteed by using a similar variational argument as in Lemma~\ref{v_variation}, resulting  a $\hat{v}$ satisfying $\hat{v}^{\prime}(0)=0$.
 We claim that $A = \N^{\prime}(w)$. It is clear that $A$ is linear. With $\hat{v}^{\prime\prime} = \gamma\hat{v} + 3v^2\hat{v} - \hat{w}$, it follows 
 from Lemma~\ref{v_embedding} that there exists a constant $C_0 = C_0(\gamma, \|3 v^2 \|_{H^1})$ such that 
 $\| \hat{v} \|_{H^3} \leq C_0 \| \hat{w} \|_{H^1}$. Hence $A$ is a bounded operator. To finish our proof, it suffices to check that 
\begin{equation} \label{N_differentiable}
\| \N(w + \hat{w}) - \N w - A\hat{w}  \|_{H^3}= o(\|\hat{w}\|_{H^1})
\end{equation}
for any $\hat{w} \in H^1$ with norm at most $1$. Let $\tilde{v} = \N(w+ \hat{w}) - \N w$. Since both $(w + \hat{w}, v+ \tilde{v})$ and $(w, v)$ satisfy (\ref{FN_nonlinear_steady}b), we have
\begin{equation*}
\begin{cases}
(v+ \tilde{v})^{\prime\prime} -\gamma (v + \tilde{v}) - (v + \tilde{v})^3 = -(w+ \hat{w}), \\
\phantom{(v+\tilde{v})^{\prime\prime} - \gamma(+\tilde{v}} \,\,\, v^{\prime\prime}- \gamma v - v^3 = -w.
\end{cases}
\end{equation*}
Subtracting from one another yields
\begin{equation} 
\tilde{v}^{\prime\prime} - \gamma\tilde{v} - 3v^2\tilde{v} - 3v\tilde{v}^2 - \tilde{v}^3 = -\hat{w}, \label{v_hat}
\end{equation}
and we subtract (\ref{v_hat1}) from (\ref{v_hat}) to get
\begin{equation*} \label{w_xx}
(\tilde{v}-\hat{v})^{\prime\prime} - (\gamma + 3v^2) (\tilde{v} - \hat{v}) = \tilde{v}^3 + 3v\tilde{v}^2.
\end{equation*}
By applying Lemma~\ref{v_embedding} and the estimate from Lemma~\ref{u_embedding}, there exists a positive constant $C_0 =C_0 (\gamma,\| 3v^2 \|_{H^1})$ such that
\begin{equation*}
\|\tilde{v} - \hat{v}\|_{H^3} \leq \sqrt{5} \, C_0\|\tilde{v}+3v\|_{H^1}\|\tilde{v}\|^2_{H^1} \;.
\end{equation*}
Since $\gamma\tilde{v} + 3v^2\tilde{v} + 3v\tilde{v}^2 + \tilde{v}^3 = \left( \gamma+ (\tilde{v}+ 3v/2)^2 + 3 v^2/4    \right)  \tilde{v}$,
the weak formulation of \eqref{v_hat} implies $\|\tilde{v}\|_{H^1} \leq \max\{1, 1/\gamma\}\|\hat{w}\|_{H^1}$. Together with $\|v\|_{H^1} \leq \max\{1, 1/\gamma\} \|w \|_{H^1}$ from Lemma~\ref{w_and_Nw}, we finally have 
\begin{align*}
\|\tilde{v} - \hat{v}\|_{H^3} &\leq \sqrt{5} \, C_0 \left( \|\tilde{v}\|_{H^1} + \|3v\|_{H^1}\right)\|\tilde{v}\|^2_{H^1} \\
&\leq \sqrt{5} \, C_0\max\{1, 1/\gamma^3 \}\left(\|\hat{w}\|_{H^1} +3\|w\|_{H^1} \right)\|\hat{w}\|^2_{H^1} \\
&\leq C_1 \|\hat{w}\|^2_{H^1}
\end{align*}
for some $C_1 =C_1(\gamma,  \|w\|_{H^1}),$ which implies (\ref{N_differentiable}) as desired.
\end{proof}

\begin{lem} \label{N_increasing}
Suppose $w_1, w_2 \in H^1(0,\infty)$ are distinct with $w_1 \geq w_2$, then $\N w_1 > \N w_2$.
\end{lem}

\begin{proof}
Let $w_1, w_2 \in H^1(0, \infty)$ with $w_1 \geq w_2$. Then $v_1 = \N w_1$ and $v_2 = \N w_2$ satisfy $v_1^{\prime\prime} - \gamma v_1 - v_1^3 + w_1 = 0$ and $v_2^{\prime\prime} - \gamma v_2 - v_2^3 + w_2 = 0$, respectively. By subtracting the two equations, we obtain 
\begin{align} 
(v_1 - v_2)^{\prime\prime} - \gamma (v_1 - v_2) - (v_1^2 + v_1v_2 + v_2^2)(v_1 -v_2) &= -(w_1 - w_2) \leq 0. \label{max_argument}
\end{align}
Let $z= v_1 - v_2 \in H^1(0, \infty)$. Then $z^{\prime}(0) = 0$ and $z \ra 0$ as $x \ra \infty$. Since $v_1^2+v_1v_2+v_2^2 \geq 0$, the maximum principle 
is applicable to \eqref{max_argument} and $z$ cannot attain an interior non-positive minimum unless $z \equiv 0$. The last 
possibility is excluded as $w_1$ and $w_2$ are distinct.

Suppose $z(0) \leq 0$, then $z'(0)>0$ as a result of the Hopf lemma. This is a contradiction and hence $z(0)>0$.
Coupled with the absence of  a non-positive interior minimum point, we see that $z>0$ everywhere and the proof of the lemma is complete.
\end{proof}

\begin{lem} \label{difference}
Suppose $w_1, w_2 \in L^2(0,\infty)$, then 
\[
\| \N w_2 - \N w_1 \|_{H^1(0,\infty)} \leq \max\{1, 1/\gamma\}\|w_2 - w_1 \|_{L^2(0, \infty)}.
\]
\end{lem}
\begin{proof}
 The same proof as in Lemma~\ref{N_increasing} leads us to \eqref{max_argument} (but without $\leq 0$ at the end). 
Now multiply by $v_1-v_2$ and integrate over the interval $(0,\infty)$. 
\end{proof}

\section{A variational formulation} \label{sec_varForm}
\setcounter{equation}{0}
\renewcommand{\theequation}{\thesection.\arabic{equation}}

In this section, we introduce a variational formulation that corresponds to the system \eqref{FN_nonlinear_steady} or, equivalently, to
\begin{equation}
du^{\prime\prime} + f(u) - \N u = 0.	 \label{EL_eqn}
\end{equation}
Consider the functional $\hat{J}: H^1(0, \infty) \ra \R$ defined by 
\begin{equation*}
\hat{J}(w) = \integral \Big \{ \frac{d}{2}w^{\prime2} + \frac{1}{2}w\N w + F(w) + \frac{1}{4}(\N w)^4 \Big \} \, dx, 
\end{equation*}
where 
\begin{equation*}
F(\xi) = -\int_{0}^{\xi} f(\eta) \, d\eta = \frac{\xi^4}{4}-\frac{(1+\beta)\xi^3}{3} +  \frac{\beta\xi^2}{2}.
\end{equation*}
Let $0 < \beta_1 < 1 < \beta_2$ such that $F(\beta_1) = F(\beta_2) = 0$. We will first verify that \eqref{EL_eqn} is the Euler-Lagrange equation associated with $\hat{J}$.

\begin{lem} 
The functional $\hat{J}$ is well defined for all $w \in H^1(0, \infty)$.
\end{lem}

\begin{proof}
Let $w \in H^1(0, \infty)$. Then $\| w \|_{L^{\infty}{(0, \infty)}} \leq \sqrt{2} \| w \|_{H^1{(0, \infty)}}$ and
$\|\N w\|_{L^{\infty}(0, \infty)} \leq \sqrt{2}\,\|\N w\|_{H^1(0, \infty)}$ by Lemma~\ref{u_embedding}. For a fixed $w$, there is a positive constant
$C_w$, which depends on $\|w\|_{L^{\infty}(0, \infty)}$, such that $|F(\xi)| \leq C_w\xi^2$ for $|\xi| \leq \|w\|_{L^{\infty}(0, \infty)}$. Together with Lemma~\ref{w_and_Nw}, we obtain
\begin{align*}
|\hat{J}(w)| &\leq \frac{d}{2}\|w\|^2_{H^1} + \frac{1}{2}\|w\|_{L^2}\|\N w\|_{L^2} + C_w\|w\|^2_{L^2} + \frac{1}{4}\|\N w\|^2_{L^\infty}\|\N w\|_{L^2}^2 \\
&< \frac{d}{2}\|w\|^2_{H^1} + \frac{1}{2}\max\{1,1/\gamma\}\|w\|_{L^2}^2 + C_w\|w\|^2_{H^1} + \frac{1}{2}\|\N w\|_{H^1}^4 \\
&< \infty.  
\end{align*} 
\end{proof}

\begin{lem} \label{rewriteF} 
Let $v = \N w$. Then
\[ \integral \Big \{ \frac{1}{4}v^4 + \frac{1}{2}wv  \Big \} \, dx = \integral \Big\{ -\frac{1}{2}v^{\prime2} - \frac{\gamma}{2} v^2 - \frac{1}{4}v^4 + wv  \Big\} \, dx. \]
\end{lem}

\begin{proof}
For $(w, v)$ satisfies (\ref{FN_nonlinear_steady}b) weakly, 
\begin{equation*}
\integral \frac{1}{2} \left(-v^{\prime}\varphi^{\prime} - \gamma v\varphi - v^3\varphi + w\varphi \right)  \, dx = {0,} \, \quad \forall \varphi \in H^1(0, \infty).
\end{equation*}
We choose $\varphi = v$ and add $\integral \left( \frac{1}{4}v^4 + \frac{1}{2}wv \right) dx$ on both sides to get the result.
\end{proof}

\begin{lem}
If $u_0 \in H^1(0, \infty)$ is a critical point of $\hat{J}$, then $(u_0, \N u_0)$ is a weak solution of \eqref{FN_nonlinear_steady}.
\end{lem}

\begin{proof}
Given any $w \in H^1(0, \infty)$, define $v = \N w$. By Lemma~\ref{rewriteF} we can write
\[ \hat{J}(w) = \integral \Big \{ \frac{d}{2}w^{\prime2} -\frac{1}{2}v^{\prime2} - \frac{\gamma}{2} v^2 - \frac{1}{4}v^4 + wv + F(w)  \Big\} \, dx.  \]
With $\hat{v} = \N^{\prime}(w)\hat{w}$, the Fr\'echet derivative of $\hat{J}$ is 
\begin{equation} \label{J_derivative}
\hat{J}'(w)\hat{w} = \integral  \Big\{  dw'\hat{w}^{\prime}  - v'\hat{v}^{\prime} - \gamma v \hat{v} - v^3\hat{v} + w\hat{v} + v\hat{w}  - f(w)\hat{w}  \Big\} \, dx.
\end{equation}
Since $\hat{J}^{\prime}(u_0)\hat{w} = 0$ and $v_0 = \N u_0$ satisfies (\ref{FN_nonlinear_steady}b), the equation (\ref{J_derivative}) becomes
\[\integral  \{  d{u}_0^{\prime}\hat{w}^{\prime} - f(u_0)\hat{w}  + v_0\hat{w} \} \, dx = 0,  \]
which implies that $(u_0, v_0)$ satisfies (\ref{FN_nonlinear_steady}a) weakly. 
\end{proof}

\begin{remark} \label{u0_bndry}
The critical point $u_0$ of $\hat{J}$ satisfies the natural boundary condition $u_0^{\prime}(0) = 0$.
\end{remark}

To find a standing pulse solution of \eqref{FN_nonlinear_steady}, we now consider a minimizing problem for $\hat{J}$. Define a class of admissible functions $\A$ as  
\begin{align} \label{admissible_set}
\A \equiv \{ w &\in H^1(0, \infty) : \beta \leq w(0) \leq 1; \text{ there exist } 0 \leq x_1 < x_2 \leq \infty \text{ such that } \beta \leq w \leq 1 \text { on } [0, x_1], \nonumber\\
&0 \leq w \leq \beta \text{ on } (x_1, x_2] \text{ and } -(M + 1) \leq w \leq 0 \text{ on } (x_2, \infty) \},
\end{align}
where $M = M(\gamma)$ is a constant such that $f(\xi) \geq 1 + 1/\gamma$ for all $\xi \leq -M$. We note that the initial condition $\beta \leq w(0) \leq 1$ is vacuous if $x_1 = 0$.
Without any constraint we expect there is no global minimizer of $\hat{J}$, a fact demonstrated in the work of \citet{Chen:2012}.
We therefore restrict our attention to $J \equiv \hat{J}|_{\A}$ for a minimizer. In what follows, let us refer to the terms $\integral \frac{d}{2} w^{\prime \,2}  \,dx$, $\integral F(w) \,dx,$ and $\integral \left( \frac{1}{2}w\N w +\frac{1}{4}(\N w)^4 \right) \,dx$ as the gradient term, potential term, and nonlocal term of $J$, respectively. 

The presence of the nonlocal term imposes a difficulty in showing the existence of a minimizer. To attain a minimizer in the next section, we discuss some estimates of the nonlocal term that will be useful.
\begin{lem} \label{v_upperbnd}
Let $w \in \A$. Then $-(M + 1) \leq \N w \leq 1$.
\end{lem}

\begin{proof}
Set $v = \N w$ and $\bar{v} =1$. Since $w \leq 1$, 
\begin{equation*} 
\bar{v}^{\prime\prime} - \gamma \bar{v} - \bar{v}^3 = -\gamma - 1 \leq -w.
\end{equation*}
By subtracting $v^{\prime\prime}  -\gamma v - v^3 = -w$ from above, 
\begin{equation*} 
(\bar{v} - v)^{\prime\prime} - \gamma (\bar{v} - v) - (\bar{v}^2 + \bar{v}v + v^2)(\bar{v} - v) \leq 0. 
\end{equation*}
Let $z = \bar{v} - v$. The same maximum principle argument stated after \eqref{max_argument} enables us to conclude that $z \geq 0$ everywhere, i.e. $v \leq 1$. {Similarly for the lower bound, set $\underline{v} = -(M+1)$ and observe that, since $w \geq -(M + 1)$,
\begin{align*}
(v-\underline{v})^{\prime\prime} - \gamma (v - \underline{v}) - (v^2 + \underline{v}v + \underline{v}^2)(v - \underline{v}) = -(w + \gamma(M +1) + (M+1)^3) \leq 0.
\end{align*}
The argument as before leads to $v - \underline{v} \geq 0$.}
\end{proof}

In next, we use a comparison to obtain an estimate of $\N$. Consider the following linear equations
\begin{equation} \label{linear_eqns}
\begin{cases}
   V^{\prime\prime}  - \gamma V \,\, + w = 0,   \\
   V_0^{\prime\prime} - (\gamma + 1)V_0 + w = 0,
\end{cases}
\end{equation}
with zero Neumann boundary conditions at $x=0$ for a fixed $w \in L^2(0, \infty)$. By solving (\ref{linear_eqns}a), we write
\begin{equation} \label{operator_L}
V(x) = \L w(x) = \int_{0}^{\infty} G(x,s)w(s) \, ds,
\end{equation}
where $\L: L^2(0, \infty) \ra L^2(0, \infty)$ is a linear operator with the Green's function 
\begin{equation*}
G(x,s)=
\begin{cases}
 	\frac{1}{\sqrt{\gamma}} e^{-\sqrt{\gamma}s} \text{  } \textrm{cosh} \, \sqrt{\gamma}x  \, , & \text{   if } x < s, \\
 	\frac{1}{\sqrt{\gamma}} e^{-\sqrt{\gamma}x} \text{  } \textrm{cosh} \, \sqrt{\gamma}s \, , & \text{   if } x > s.
	\end{cases}
\end{equation*}
It can be verified that $\integral w_1 \L w_2 \, dx= \integral w_2 \L w_1 \, dx$ for any $w_1, w_2 \in L^2(0, \infty),$ i.e. $\L$ is self-adjoint 
with respect to the $L^2$ inner product. Moreover, a direct calculation shows
\begin{align} 
\L w(x) &=  \int_{0}^{x}  \frac{1}{\sqrt{\gamma}} e^{-\sqrt{\gamma}x} \textrm{cosh} \, (\sqrt{\gamma}s) \, w(s) \, ds 
+ \int_{x}^{\infty} \frac{1}{\sqrt{\gamma}} e^{-\sqrt{\gamma}s} \textrm{cosh} \, (\sqrt{\gamma}x) \, w(s) \, ds \notag \\
&\leq  \frac{1}{\sqrt{\gamma}} e^{-\sqrt{\gamma}x} \int_{0}^{x}  \textrm{cosh} \, (\sqrt{\gamma}s) \, ds + \frac{ \textrm{cosh} \, (\sqrt{\gamma}x)}{\sqrt{\gamma}} \int_{x}^{\infty} e^{-\sqrt{\gamma}s} \, ds \notag \\
&= \frac{1}{\gamma}. \label{expand_L}
\end{align}
Similarly, for (\ref{operator_L}b), we can set $\L_0 = \Big( (\gamma+1) -\frac{d^2}{dx^2} \Big)^{-1}$ and write
\begin{equation} \label{operator_L*}
V_0(x) = \L_0 w (x) = \integral G_0(x, s) w(s) \, dx,
\end{equation}
where
\begin{equation*}
G_0(x,s)=
\begin{cases}
 	\frac{1}{\sqrt{\gamma+1}} e^{-\sqrt{\gamma+1}\, s} \text{  } \textrm{cosh} \,( \sqrt{\gamma+1}\, x  )\, , & \text{   if } x < s, \\
 	\frac{1}{\sqrt{\gamma+1}} e^{-\sqrt{\gamma+1}\, x} \text{  } \textrm{cosh} \, (\sqrt{\gamma+1}\, s ) \, , & \text{   if } x > s.
	\end{cases}
\end{equation*} 

\begin{lem} \label{comparison}
For a non-negative, non-trivial function $w \in \A$,  
\begin{equation*}
{{0 <}} \, \L_0w \leq \N w \leq \L w.
\end{equation*}
\end{lem}

\begin{proof}
{Let $V_0 = \L_0 w$ and $V = \L w$. The positivity of the Green's function $G_0$ implies that {$V_0 > 0$.}} Since $v = \N w$ satisfies (\ref{FN_nonlinear_steady}b), we have $v''- (\gamma+v^2) v \leq 0$ so that $v>0$ by the maximum principle. In addition
\begin{equation} \label{rewrite_v}
v^{\prime\prime} - (\gamma + 1)v + w = v^3 - v.
\end{equation}   
By subtracting (\ref{linear_eqns}b) from \eqref{rewrite_v} and using $0 \leq v \leq 1$, we obtain
\begin{align*}
(v-V_0)^{\prime\prime} - (\gamma+1)(v - V_0) = v^3 - v \leq 0.
\end{align*}
We now conclude $V_0 \leq v$ using the maximum principle.
The proof for $v \leq V$ is similar.
\end{proof}

\begin{lem} \label{nonlocal_pos}
If $w \in H^1$, then 
\begin{equation*} 
\integral w \N w \, dx \geq 0, 
\end{equation*}
and for any $w_1, w_2 \in H^1$, 
\begin{equation*}
\integral(w_1-w_2)(\N w_1 - \N w_2) \, dx \geq 0.
\end{equation*}
\end{lem}

\begin{proof}
Let $v = \N w$, then $(w, v)$ satisfies (\ref{FN_nonlinear_steady}b). Multiplying (\ref{FN_nonlinear_steady}b) by $v$ and integrating by parts gives $\integral w \N w \, dx = \integral ( {v^{\prime}}^2 + \gamma v^2 + v^4 ) \, dx \geq 0$. Next let $v_1 = \N w_1$ and $v_2 = \N w_2$. Subtracting the equations (\ref{FN_nonlinear_steady}b) for $v_1$ and $v_2$ from one another, we get $(v_1- v_2)^{\prime\prime} - \gamma(v_1 - v_2) - p(x)(v_1 - v_2) = -(w_1 - w_2)$ where $p=v_1^2+v_1 v_2 +v_2^2 \geq 0$.
The same integration by parts argument yields the next inequality.
\end{proof}

\begin{lem} \label{rewrite_nonlocal} 
Let $w = f-g$ with $f\equiv \max \{w, 0\} \geq 0$ and $g \,\geq \,0$ being its positive and negative parts, respectively. Then 
\begin{align*}  
\integral w\N w \, dx  \geq \integral (f-g)(\N f- \N g) \, dx - {4}\integral \L f \, \L g \, dx,  
\end{align*}
where $\L$ is the linear operator defined in \eqref{operator_L}.
\end{lem}

\begin{proof} Let $v_f = \N f$, $v_g = \N g$ and $v_{f-g} =\N (f-g)$. Notice from (\ref{FN_nonlinear_steady}b) that 
$\N u = \L u - \L ((\N u)^3)$ for any $u$. Hence
\begin{align*}
\integral w\N w \, dx &= \integral (f-g)\N(f-g) \, dx  \\
&= \integral (f-g)(\L f - \L g - \L\vfg^3) \, dx  \\
&= \integral (f-g)(v_f + \L v_f^3 - v_g - \L v_g^3 - \L\vfg^3) \, dx.
\end{align*}
Since $\L$ is self-adjoint with respect to the $L^2$ inner product, 
\begin{equation}
\integral w \N w \, dx = \integral (f-g)(v_f - v_g) \, dx - \integral (v_g^3 - v_f^3 + \vfg^3)\L(f-g) \, dx. \label{nonlocal_upp}
\end{equation}
{It remains to show that $\integral (v_g^3 - v_f^3 + \vfg^3)\L(f-g) \, dx \leq  {4} \integral \L f \, \L g \, dx$.} {Recall from Lemma~\ref{comparison} that $\L f$ and $\L g$ are non-negative. Since $\N$ is non-decreasing by Lemma~\ref{N_increasing}, we have $-v_g \leq v_{f-g} \leq v_f$. Then
$\lvert v_{f-g} \rvert \leq \max\{v_f, v_g\}$, which implies that $\vfg^3 \leq v_f^3 + v_g^3$. Observe that} 
\begin{align*}
\integral (v_g^3 - v_f^3 + \vfg^3 )\L(f-g) \, dx &= \integral  \{(v_g^3 - v_f^3)(\L f -\L g) + \vfg^3 (\L f - \L g) \} \, dx \\
&\leq \integral  \{(v_g^3 - v_f^3)(\L f -\L g) + (v_g^3 + v_f^3) (\L f + \L g) \} \, dx \\
&= 2\integral (v_g^3\L f +  v_f^3\L g) \, dx.  
\end{align*}
{Together with $0 \leq v_f \leq 1$, $0 \leq v_g \leq 1$, $v_g \leq \L g$ and $v_f \leq \L f$,} 
\begin{align*}
\integral (v_g^3 - v_f^3 + \vfg^3)\L(f-g) \, dx \leq {4} \integral \L f \, \L g \, dx.
\end{align*}
\end{proof}	

\begin{lem} \label{Nu_convergence}
Suppose there exists a sequence $\{\un\}_{n=1}^{\infty}$ such that $\un \rightharpoonup  u_0$ weakly in $H^1(0, \infty)$ with $\|\un\|_{H^1(0, \infty)} < \infty$. Then $\integral u_0 \N\un \, dx \ra \integral u_0\N u_0 \, dx$. 
\end{lem}

\begin{proof}
Let $\ep > 0$ be given. Since $u_0 \in H^1(0, \infty)$, there exists a large $a > 0$ such that 
\begin{equation}
\int_{a}^{\infty} u_0^2 \, dx \leq \ep. \label{u0_ep}
\end{equation}
By compactness we can find a subsequence of $\{\un\}_{n=1}^{\infty}$, still denoted by $\{\un\}$, such that 
$\un \to u_0$ in $L^2(0,a)$.
In conjugation with \eqref{u0_ep} and Lemma~\ref{difference},
for any arbitrary $\ep > 0$ there is a $N_0 > 0$ such that whenever $n \geq N_0$,
\begin{align*}
\integral \lvert u_0\N\un - u_0\N u_0\rvert \, dx &= \int_{0}^{a} \lvert u_0 \rvert \lvert \N\un - \N u_0\rvert \, dx + \int_{a}^{\infty} \lvert u_0 \rvert \lvert \N\un - \N u_0\rvert \, dx \\
&\leq \ep \, \|u_0\|_{L^2(0, a)} +\ep\,\| \N\un - \N u_0\|_{L^2(a, \infty)} \;.
\end{align*}
As $\| \N \un \|_{L^2(0,\infty)}$ is bounded because of Lemma~\ref{w_and_Nw}, our result follows.
\end{proof}

\section{Estimates for a Minimizing Sequence} \label{sec_seqEstimate}
\setcounter{equation}{0}
\renewcommand{\theequation}{\thesection.\arabic{equation}}

To extract a minimizer from a minimizing sequence $\{\wn\}_{n=1}^{\infty}$ of $J$, we need some a priori estimates on the sequence.
\begin{lem} \label{J_below}
There exists a positive constant $d_0$, which may depend on $\gamma$, such that if $d \leq d_0$, there are a $q_0 \in \A$ and 
a positive constant $M_0 = M_0(\gamma) $, which is independent of $d$, such that $J(q_0) \leq -M_0$.
\end{lem}
\begin{proof}
\noindent Let $0 < a < b$ be constants whose values will be assigned later. 
We first impose a constraint $b-a \leq 1$.
Define a piecewise linear function
\begin{equation*}
q_0(x) \equiv
\begin{cases}
 1 \, , & \text{   if } 0 \leq x \leq a, \\
 \frac{b-x}{b-a}  \, , & \text{   if } a \leq x \leq b,  \\
 0 \, , & \text{   if } x \geq b \,.
\end{cases}
\end{equation*}
Let $v = \N q_0$ and $V = \L q_0$, where $\L$ is the linear operator defined in \eqref{operator_L}. Since $(q_0, v)$ satisfies (\ref{FN_nonlinear_steady}b), we obtain $\integral v^4 \, dx \leq \integral q_0 v \,dx$ from the weak formulation of (\ref{FN_nonlinear_steady}b). Then,
\begin{align}
J(q_0) = \integral \left \{ \frac{d}{2}{q'_0}^2 + \frac{1}{2}q_0v + F(q_0) + \frac{1}{4} {v}^4 \right \} \, dx \leq \integral \left \{ \frac{d}{2}{q'_0}^2 + F(q_0) + \frac{3}{4}q_0v \right \} \, dx. \label{J_bound}
\end{align}
A direct computation yields
\begin{equation*} \label{gradient_J}
\integral \frac{d}{2}{q'_0}^{2} \,dx = \frac{d}{2(b-a)} 
\end{equation*}
and, with $F(\xi) = \xi^4/4 - (1+\beta)\xi^3/3 +\beta\xi^2/2,$
\begin{equation*} \label{integral_J}
\integral F(q_0) \, dx = -\frac{(1-2\beta)}{12}a + (b-a)\left \{\frac{1}{20}  - \frac{(1+\beta)}{12} + \frac{\beta}{6} \right \}. 
\end{equation*}
{For the nonlocal term,} it follows from Lemma~\ref{u_embedding} and Lemma~\ref{w_and_Nw} that
\begin{align*}
\integral  q_0v \, dx \leq \sqrt{2} \|v\|_{H^1}\|q_0\|_{L^1} \leq \sqrt{2}\max{\{1, 1/\gamma}\}\|q_0\|_{L^2}\|q_0\|_{L^1} \;.
\end{align*}
Then by computing the norms of $q_0$ directly, we obtain
\begin{align*}
\integral \frac{3}{4}q_0 v \, dx &\leq \frac{3\sqrt{2}}{4} \max{\{1, 1/\gamma}\} \left (a + \frac{1}{3} (b-a) \right)^{\frac{1}{2}}\left(a + \frac{1}{2}(b-a)\right) \notag \\
&\leq \frac{3\sqrt{2}}{4} \left(1 + \frac{1}{\gamma} \right) \left ( a + \frac{1}{2}(b-a) \right)^{\frac{3}{2}}. \label{nonloc_J}
\end{align*}
Take $d_0 = (b-a)^2$, and let $d \leq d_0$. Plugging the gradient, potential and nonlocal terms into \eqref{J_bound},
\begin{equation*}
J(q_0) \leq (b-a) \left\{  \frac{11}{20}  - \frac{(1+\beta)}{12} + \frac{\beta}{6} \right\} - \frac{1-2\beta}{12}a + \frac{3\sqrt{2}}{4}\left(1 + \frac{1}{\gamma} \right) \left ( a + \frac{1}{2}(b-a) \right)^{3/2}. 
\end{equation*}
Let $C_0 \equiv  \frac{11}{20}  - \frac{(1+\beta)}{12} + \frac{\beta}{6}$ and note that $C_0 \geq  7/15$, the lower limit being attained when $\beta=0$. Assume
$a \leq 24 C_0/(1-2\beta)$ and $(b-a) \leq  \frac{(1-2\beta)}{24C_0}a \leq 1$, then  
\begin{align*}
J(q_0) & \leq - \frac{1-2\beta}{24}a + \frac{3\sqrt{2}}{4}\left(1+\frac{1}{\gamma}\right)\left(1 + \frac{1-2\beta}{48C_0} \right)^{3/2} a^{3/2} \\
 &  \leq - \frac{1-2\beta}{48}a 
\end{align*}
by choosing $a = \frac{2}{9} \left ( \frac{1-2\beta}{24} \right)^2 \left(1 + \frac{1}{\gamma} \right)^{-2} \left(1+\frac{1-2\beta}{48C_0}\right)^{-3} \leq 24\, C_0/(1-2\beta)$.  We therefore obtain
\begin{equation*}
J(q_0) \leq  - \frac{1}{9} \left (\frac{1-2\beta}{24} \right)^3 \left(1 + \frac{1}{\gamma} \right)^{-2} \left(1+\frac{1-2\beta}{48C_0}\right)^{-3}
:= - M_0.
\end{equation*}
Recall that $C_0$ is independent of $\gamma$. As $\gamma \to 0$ we see that $a$ can go to $0$, which in turn forces {$d_0=(b-a)^2 \leq \left(\frac{1-2\beta}{24C_0}\right)^2a^2 \to 0$}. 
Hence there exist a $d_0 = (b-a)^2$, which may depend on $\gamma$, and a positive constant $M_0 := M_0(\gamma)$ such that if $d \leq d_0$, {we have $J(q_0) \leq -M_0$}. 
\end{proof}

In what follows, let $\gamma_0 \equiv \frac{3\beta^2}{1-2\beta} - 1$. We remark that $\gamma_0 > 0$ for $\beta \in (\frac{1}{3}, \frac{1}{2})$.  

\begin{lem} \label{min_sequence}
If $\gamma \leq \gamma_0$ and $d \leq d_0$, then  
\begin{enumerate}[(i)] \itemsep=0.25mm
\item $\inf_{w \in \mathcal{A}}  \mathnormal{J}(w) \geq -M_1$ for some positive constant $M_1 = M_1(\gamma).$
\item Recall the definition of $x_1$ in \eqref{admissible_set}. For any minimizing sequence $\{ \wn\}$ of $J$, let $x_{1}^{(n)}$ be a corresponding value for $w^{(n)}$. By focusing on the tail of the sequence if necessary, $0 < m_2 \leq x_1^{(n)} \leq M_2 < \infty$ for all $n$, where $M_2=M_2(\gamma)$ and $m_2=m_2(\gamma)$ are positive constants which are independent
of $n$.
\item A minimizing sequence $\{w^{(n)} \}$ is uniformly bounded for all $n$ in $H^1(0,\infty)$ norm.
\end{enumerate}
\end{lem}

\begin{proof}
Let $w=f-g$ where $f \geq 0$ and $g \geq 0$ are the positive and negative parts of $w$, respectively, as in Lemma~\ref{rewrite_nonlocal},
thus we can write
\begin{equation} \label{IminusII}
\integral w\N w \, dx \geq {I - 4\; II},
\end{equation}
where $I \equiv  \integral (f-g)(\N f- \N g) \, dx$ and $II \equiv \integral \L f \, \L g \, dx.$ Since $\L_0 w \leq \N w \leq \L w$ for any $w \geq 0$ by Lemma~\ref{comparison} and $\integral g \N g \, dx \geq 0$ by Lemma~\ref{nonlocal_pos}, together with the self-adjointness of $\L$, we get
\begin{align} 
I &\geq \integral \{ f\N f - g\N f -f\N g \} \, dx \notag \\
&\geq  \integral \{ f\L_0 f - 2f\L g \} \, dx \notag \\
&\geq \beta\int_{0}^{x_1} \L_0 f \, dx - 2\int_{0}^{x_2} \L g \, dx.  \label{I_bound}
\end{align}
For $0 \leq x \leq x_1$, we use the definition of $\L_0$ in \eqref{operator_L*} to compute
\begin{align} \label{Lf_bound}
\L_0 f(x) &= \frac{e^{-\sqrt{\gamma+1}\,x}}{\sqrt{\gamma+1}} \int_{0}^{x} f(s) \textrm{cosh}(\sqrt{\gamma+1}\, s) \, ds + \frac{\textrm{cosh} (\sqrt{\gamma+1}x)}{\sqrt{\gamma+1}} \int_{x}^{x_2} f(s) e^{-\sqrt{\gamma+1}s} \, ds \notag \\
&\geq \frac{\beta e^{-\sqrt{\gamma+1}\, x}}{\sqrt{\gamma+1}} \int_{0}^{x}\textrm{cosh}(\sqrt{\gamma+1}\, s) \, ds \notag \\
&= \frac{\beta}{\gamma+1}e^{-\sqrt{\gamma+1}x}\textrm{sinh}(\sqrt{\gamma+1} \,x).
\end{align}
Similarly for $0 \leq x \leq x_2$, we obtain from the definition of $\L$ in \eqref{operator_L} that
\begin{align} \label{Lg_bound}
\L g (x) &=\frac{\textrm{cosh}(\sqrt{\gamma} x)}{\sqrt{\gamma}}  \int_{x_2}^{\infty} g(s) e^{-\sqrt{\gamma} s} \, ds \notag \\
&\leq \frac{(M+1)\textrm{cosh}(\sqrt{\gamma} x)}{\sqrt{\gamma}} \int_{x_2}^{\infty} e^{-\sqrt{\gamma} s} \, ds \notag \\
&= \frac{(M + 1)}{\gamma} e^{-\sqrt{\gamma} x_2}\textrm{cosh}(\sqrt{\gamma} x).  
\end{align}
Finally by plugging in \eqref{Lf_bound} and \eqref{Lg_bound} into \eqref{I_bound}, 
\begin{align}
I &\geq \frac{\beta^2}{\gamma+1} \int_{0}^{x_1} e^{-\sqrt{\gamma+1}\, x}\textrm{sinh}(\sqrt{\gamma+1} x) \, dx - \frac{2(M + 1)}{\gamma}e^{-\sqrt{\gamma} x_2} \int_{0}^{x_2} \textrm{cosh}(\sqrt{\gamma} x) \, dx \notag \\
&=  \frac{\beta^2}{2(\gamma+1)} \int_{0}^{x_1} (1-e^{-2\sqrt{\gamma+1}x})\, dx -\frac{2(M+1)}{\gamma^{3/2}} e^{-\sqrt{\gamma} x_2} \textrm{sinh}(\sqrt{\gamma} x_2) \notag \\
&= \frac{\beta^2}{2(\gamma+1)} \left(x_1 - \frac{1}{2\sqrt{\gamma + 1}}(1-e^{-2\sqrt{\gamma+1}x_1})  \right)- \frac{(M + 1)}{\gamma^{3/2}}\left (1-e^{-2\sqrt{\gamma}x_2} \right) \notag \\
&\geq  \frac{\beta^2}{2(\gamma+1)}x_1 - \frac{\beta^2}{4(\gamma+1)^{3/2}} - \frac{(M + 1)}{\gamma^{3/2}}\;. \label{I_estimate}
\end{align}

\noindent Next let us find an upper bound of $II$. {Recall from \eqref{expand_L} that $\L f \leq \frac{1}{\gamma}$; a similar calculation shows that $\L g \leq \frac{M+1}{\gamma}$}. Then
\begin{align} 
II \leq \frac{1}{\gamma}  \int_{0}^{x_2} \L g \,dx + {\frac{(M+1)}{\gamma}} \int_{x_2}^{\infty} \L f \, dx\;. \label{upp_bound}
 \end{align}
For $0 \leq x \leq x_2 \leq \infty$, it follows from \eqref{Lg_bound} that
\begin{align}
\int_{0}^{x_2} \L g \, dx &\leq  \frac{(M+1)}{\gamma} e^{-\sqrt{\gamma} x_2} \int_{0}^{x_2 }\textrm{cosh}(\sqrt{\gamma} x) \, dx  \notag \\
& \leq \frac{(M+1)}{2\gamma^{3/2}} \;. \label{Lg_int} 
\end{align}
At the same time when $x_2 \leq x < \infty$, we use the definition of $\L$ in \eqref{operator_L} to obtain
\begin{align*}
\L f(x)  &=  \frac{e^{-\sqrt{\gamma}x}}{\sqrt{\gamma}} \int_{0}^{x_2} f(s) \textrm{cosh}(\sqrt{\gamma}s) \, ds \\
&\leq \frac{e^{-\sqrt{\gamma}x}}{\gamma} \textrm{sinh}(\sqrt{\gamma}x_2),
\end{align*}
which implies that  
\begin{align}
\int_{x_2}^{\infty} \L f(x) \, dx  &\leq \frac{1}{\gamma}\textrm{sinh}(\sqrt{\gamma}x_2) \int_{x_2}^{\infty}e^{-\sqrt{\gamma}x} \, dx \notag  \\
&\leq \frac{1}{2\gamma^{3/2}}\;. \label{Lf_int}
\end{align}
Substituting \eqref{Lg_int} and \eqref{Lf_int} into \eqref{upp_bound},
\begin{equation} \label{II_estimate}
II \leq {\frac{(M+1)}{\gamma^{5/2}}}\;.
\end{equation}
Now using the bounds in \eqref{I_estimate} and \eqref{II_estimate} to estimate \eqref{upp_bound}, we get
\begin{equation*}
\integral w\N w \, dx \geq I - 4 \, II \geq \frac{\beta^2}{2(\gamma+1)}x_1 - \frac{\beta^2}{4(\gamma+1)^{3/2}} - \frac{(M+1)}{\gamma^{3/2}} - \frac{{4(M+1)}}{\gamma^{5/2}}\;.
\end{equation*}
Since $F \geq 0$ when $x \geq x_1$, with $\integral F(x) \, dx \, \geq F_{\min}x_1 : = F(1)\,x_1 = -\frac{(1-2\beta)}{12}x_1$, 
\begin{align}
J(w) &= \integral \Big \{ \frac{d}{2}w^{\prime2} + \frac{1}{2}w\N w + F(w) + \frac{1}{4}(\N w)^4 \Big \} \, dx \notag \\
&\geq \integral \{  F(w) + \frac{1}{2}w \N w  \} \, dx \notag \\
&\geq \left( -\frac{(1-2\beta)}{12} + \frac{\beta^2}{4(\gamma+1)}\right )x_1 - \frac{\beta^2}{8(\gamma+1)^{3/2}} - \frac{(M+1)}{2\gamma^{3/2}} - \frac{{2(M+1)}}{\gamma^{5/2}}\;. \label{J_below_ineq}
\end{align}
Observe that $ -\frac{(1-2\beta)}{12} + \frac{\beta^2}{4(\gamma+1)} > 0 $ for $\gamma < \gamma_0 = \frac{3\beta^2}{(1-2\beta)}-1$. Choosing $M_1 = \frac{\beta^2}{8(\gamma+1)^{3/2}} + \frac{(M+1)}{2\gamma^{3/2}} + \frac{{2(M+1)}}{\gamma^{5/2}}$, we establish (i).

{By Lemma~\ref{J_below} we can assume that a minimizing sequence $\wn$ satisfies $J(\wn) \leq -M_0 < 0$} by focusing on the tail of the sequence if needed. We can include the gradient term on the right hand side of \eqref{J_below_ineq}, doing so we have
\begin{align}
\frac{d}{2}\|\wn_{x}\|_{L^2}^2 + \left( -\frac{(1-2\beta)}{12} + \frac{\beta^2}{4(\gamma+1)}\right )x_1^{(n)} \leq M_1, \label{wx_bdd}
\end{align}
which implies that there is a positive constant $M_2=M_2(\gamma):= M_1 \big/ \left( -\frac{(1-2\beta)}{12} + \frac{\beta^2}{4(\gamma+1)}\right)$,  independent of $n$, such that 
$x_1^{(n)} \leq M_2$. Moreover, since the nonlocal term is non-negative and $F(\xi) \geq 0$ for $\xi \leq \beta_1$, 
\begin{align*}
{-\frac{1}{2}M_0} \geq \, J(\wn) &\geq \integral F(\wn) \, dx \\
&\geq \int_{\{x:\,\wn(x) \geq \beta_1\}} F(\wn) \, dx \\
&\geq -|F_{\min}| \, |\{x:\,\wn(x) \geq \beta_1\}|,
\end{align*}
which implies that $x_1^{(n)} \geq  \{x:\,\wn(x) \geq \beta_1\} \geq {6M_0/(1-2\beta)} :=m_2 >0$; hence there is always a non-trivial positive part of $\wn$. The proof of (ii) is complete. To show (iii), first observe from \eqref{wx_bdd} that $\|\wn_{x}\|_{L^2}$ is bounded for all $n$. Next, it follows from (ii) that
\begin{align*}
\int_{\{x:\,\wn(x) \geq \beta\}} (\wn)^2 \, dx \leq \int_{\{x:\wn(x) \geq \beta\}} 1 \, dx = x^{(n)}_1 \leq M_2.
\end{align*}
On $\{x:\wn(x) < \beta\}$, for there exists a $C_1 > 0$ independent of $d$ and $\gamma$ such that $F(\xi) \geq C_1\xi^2$,
\begin{align*}
\int_{\{x:\,\wn(x) < \beta\}} (\wn)^2 \, dx &\leq \frac{1}{C_1} \int_{\{x:\wn(x) < \beta\}} F(\wn) \, dx \\
&=  \frac{1}{C_1}\left\{ \integral F(\wn) \, dx - \int_{\{x:\wn(x) \geq \beta\}} F(\wn) \, dx  \right\} \\
&\leq \frac{1}{C_1}\left\{ J(\wn)  -  \int_{\{x:\wn(x) \geq \beta_1\}} F(\wn)  \, dx \right\} \\
&\leq \frac{1}{C_1}\{|F_{\min}|\, M_2\}.
\end{align*}
Therefore, $\|\wn\|_{L^2}$ is uniformly bounded for all $n$. This completes the proof that $\|\wn\|_{H^1}$ is uniformly bounded. 
\end{proof}

\section{Existence of a Minimizer} \label{sec_existence}
We now extract a minimizer $u_0 \in \A$ from the minimizing sequence obtained in {Section} \ref{sec_seqEstimate}. Due to the constraints imposed on the admissible set $\A$, $u_0$ may not satisfy (\ref{FN_nonlinear_steady}a) 
on the intervals where it identically equals to one of the constraints. To eliminate this possibility in the later sections, a truncation technique is used routinely in which 
we truncate $u_0$ to obtain a new function $\unew \in \A$ with $J(\unew) < J(u_0)$. With the help of the truncation technique, the constraint $-(M + 1)$ will be released in this section.

\begin{lem} \label{min_existence}
Suppose $\gamma \leq \gamma_0$ and $d \leq d_0$. Let $\{w^{(n)} \}_{n=1}^{\infty} \subset \A$ be a minimizing sequence of $J$. Then there exists
a $u_0 \in \A$ such that $\liminf J(w^{(n)}) \geq J(u_0)$. Moreover there exist $0 < x_1 < x_2 \leq \infty$ such that 
\begin{equation} \label{u0_range}
\begin{cases}
 	\beta \leq u_0(x) \leq 1 \textit{ for } x \in [0, x_1], \\
 	0 \leq u_0(x) \leq \beta \textit{ for } x \in [x_1, x_2], \\
	-(M + 1) \leq u_0(x) \leq 0 \textit{ for } x \in [x_2, \infty) \textit{ if } x_2 < \infty
	\end{cases}
\end{equation} 
{with $m_2 \leq x_1 \leq M_2$.}
\end{lem}

\begin{proof}
By Lemma~\ref{min_sequence} there is a minimizing sequence $\{w^{(n)} \}_{n=1}^{\infty}$ such that $\lim J(w^{(n)}) = \inf_{w\in\A} J(w)$ with $\|w^{(n)}\|_{H_1(0, \infty)}$ uniformly bounded in $n$; this sequence is therefore compact in the weak topology. By choosing a subsequence, still denoted by $\{w^{(n)} \}$, there exists a $u_0 \in H^1(0, \infty)$ such that $w^{(n)} \rightharpoonup u_0$ weakly in $H^1(0, \infty)$ and strongly in $L_{loc}^\infty(0, \infty)$. As a consequence of Lemma~\ref{min_sequence}, \eqref{u0_range} holds with $m_2 \leq x_1 \leq M_2$ and $u_0 \in \A$.

Next we show that the weakly convergent subsequence satisfies $\liminf J(w^{(n)}) \geq J(u_0)$. The weak convergence in $H^1$ implies that
\begin{equation} \label{grad_conv}
 \liminf \integral ({w}^{(n)\prime})^2 \, dx \geq \integral (u_0^{\prime})^2 \, dx.
 \end{equation}
Since $\wn \ra u_0$ on $[0, x_1]$ and $\|F(\wn)\|_{L^{\infty}(0, \infty)} < \infty$ for $\wn \in \A$, 
\begin{equation*}
\lim \int_{0}^{x_1}F(w^{(n)}) \, dx = \int_{0}^{x_1} F(u_0) \, dx.
\end{equation*}
Moreover since $\wn \leq \beta$ on $[x_1, \infty)$, Fatou's lemma implies that 
\begin{equation*}
\liminf \int_{x_1}^{\infty} F(\wn) \, dx \geq \int_{x_1}^{\infty} F(u_0) \, dx.
\end{equation*} 
Therefore we conclude that
\begin{equation} \label{potential_conv}
\liminf \integral F(\wn) \, dx \geq \integral F(u_0) \, dx.
\end{equation}

It remains to treat the nonlocal term.  
With $w_1 = \wn$ and $w_2 = u_0$, it follows from Lemma~\ref{nonlocal_pos} that
\[ \integral (\wn\N\wn + u_0\N u_0) \, dx \geq  \integral (u_0\N\wn + \wn\N u_0 ) \, dx. \]
Since $\integral u_0 \N\wn \, dx \ra \integral u_0\N u_0 \, dx$ by Lemma~\ref{Nu_convergence}
and $\integral \wn\N u_0 \, dx$ goes to the same limit because $\wn \rightharpoonup u_0$ weakly in $L^2$, 
\begin{equation*}
 \liminf \integral \wn\N\wn \, dx + \integral u_0\N u_0 \, dx \geq 2\integral u_0\N u_0 \, dx.
\end{equation*}
Then it is clear that
\begin{equation} \label{nonlocal_conv}
\liminf \integral \left(\wn\N\wn + (\N\wn)^4\right)  \, dx \geq \integral \left( u_0\N u_0 + (\N u_0)^4 \right) \, dx.
\end{equation}
Combining \eqref{grad_conv}, \eqref{potential_conv} and \eqref{nonlocal_conv}, $\liminf J(w^{(n)}) \geq J(u_0)$ follows immediately. Therefore $u_0$ is a minimizer of $J$ satisfying $J(u_0) = \inf_{\A} J$.
\end{proof}

\begin{lem} \label{nonlocal_difference}
Let $u_0$ be changed to $\unew \in A$. Then the change in the nonlocal term is
\begin{align*}
&\integral \left \{ \left ( \frac{1}{2} \unew \N\unew + \frac{1}{4}(\N\unew)^4 \right) - \left ( \frac{1}{2}u_0\N u_0 + \frac{1}{4}  (\N u_0)^4 \right) \right \} \, dx \notag \\
&\quad = \frac{1}{2}\integral (\unew - u_0)(\N\unew + \N u_0) \, dx +  \frac{1}{4} \integral(\N\unew + \N u_0)(\N \unew - \N u_0)^3 \, dx. 
\end{align*}
Moreover 
\begin{equation*}
\left \lvert \frac{1}{4} \integral(\N\unew + \N u_0)(\N \unew - \N u_0)^3  \, dx  \right \rvert 
 \leq { \frac{(M +1)^2}{4} \max \{1, \frac{1}{\gamma^2} \}}\, \integral(\unew - u_0)^2 \, dx.
\end{equation*}
\end{lem}

\begin{proof}
Set $v_0 = \N u_0$ and $v_{new} = \N\unew$. Then 
\begin{equation}
\integral \left(\vnew^{\prime\prime} - \gamma\vnew - \vnew^3 \right)v_0\, dx = \integral -\unew v_0\, dx, \label{vnew_weak}
\end{equation}
\begin{equation}
\integral \left(v_0^{\prime\prime} - \gamma v_0 - v_0^3 \right)\vnew \, dx = \integral -u_0\vnew  \, dx. \label{v0_weak}
\end{equation}
After integrating by parts each equation, we subtract one from the other to get
\begin{equation}
\integral (u_0\vnew - \unew v_0) \, dx =\integral (v_0^3\vnew - \vnew^3v_0) \, dx. \label{lin_to_cub}
\end{equation}
Observe that
\begin{align}
&\integral \left ( \frac{1}{2} \unew \vnew + \frac{1}{4}\vnew^4 - \frac{1}{2}u_0 v_0 - \frac{1}{4} v_0^4 \right)  \, dx \notag \\
&\quad = \integral  \left \{ \frac{1}{2}(\unew - u_0)(\vnew + v_0) + \frac{1}{2}(u_0\vnew - \unew v_0) + \frac{1}{4}(\vnew^4 - v_0^4) \right \} \, dx. \label{nonlocal_estim}
\end{align}
For the last two term in the integral, it follows from \eqref{lin_to_cub} that
\begin{align*}
& \integral \left \{ \frac{1}{2}(u_0\vnew - \unew v_0) + \frac{1}{4}(\vnew^4 - v_0^4) \right \} \, dx  \\
&\quad =  \integral \left \{\frac{1}{2}v_0\vnew(v_0^2 - \vnew^2) + \frac{1}{4}(\vnew^2 + v_0^2)(\vnew^2 - v_0^2) \right \} \, dx  \\
&\quad = \integral \frac{1}{4} (\vnew + v_0)(\vnew - v_0)^3 \, dx
\end{align*}
and therefore, our first inequality holds. To show the next inequality, note that $\| \vnew -v_0 \|_{L^2} \leq \max \{1, 1/\gamma \} \; \| \unew -u_0 \|_{L_2}$ from Lemma~\ref{difference}. Together with  $\|\vnew\|_{L^{\infty}} \leq M+1$ and $\|v_0\|_{L^{\infty}} \leq M+1$ from  Lemma~\ref{v_upperbnd}, we obtain
\begin{align*}
\left \lvert \frac{1}{4} \integral(\vnew + v_0)(\vnew - v_0)^3 \, dx \right \rvert  & \leq  \, \frac{1}{4} \integral \max\{ \vnew^2, v_0^2\} \,(\vnew - v_0)^2 \, dx \\
&\leq { \frac{(M +1)^2}{4} \max \{1, \frac{1}{\gamma^2} \}}\, \integral(\unew - u_0)^2 \, dx.
\end{align*}
\end{proof}

\begin{lem} \label{M1_const}
Let $d \leq d_0$ and $u_0$ be a minimizer obtained in Lemma~\ref{min_existence}. Then $\min u_0 \geq -M$.
\end{lem}

\begin{proof}
Suppose $\min u_0 < -M$. Take a small positive $\delta < 1$ so that $-M \geq \min u_0  +\delta$.
Consider a truncated function 
\begin{equation*}
\unew =
\begin{cases}
u_0,  \quad \quad \qquad\,  \text{ if } u_0 \geq \min u_0 + \delta,\\
\min u_0 + \delta, \;\; \text{ if } u_0 < \min u_0 + \delta,
\end{cases}
\end{equation*}
so that $\unew - u_0 \equiv p(x) \leq \delta$. Then $\unew \in \A$. 
We now define a positive constant $M_\delta : = -(\min u_0 + \delta)$ to simplify the notation.
Since the energy associated with the gradient term decreases by the change,
\begin{eqnarray*}
J(\unew) - J(u_0) &<& \int_{\{x: u_0 \leq -M_\delta\}} \{F(\unew) - F(u_0)\} \,dx \\
&&\quad + \integral \left \{ \frac{1}{2}\left( \unew \N \unew - u_0\N u_0 \right) + \frac{1}{4} \left ( (\N\unew )^4 - (\N u_0)^4 \right )\right\} \, dx
\end{eqnarray*}
and applying Lemma~\ref{nonlocal_difference} gives
\begin{eqnarray*}
J(\unew) - J(u_0) &<&  \int_{\{x: u_0 \leq -M_\delta\}} \{F(\unew) - F(u_0) + \frac{p}{2}\left(\N\unew + \N u_0 \right)\} \, dx \\
&&\quad + \frac{1}{4}\integral \left( \N\unew + \N u_0\right)\left(\N\unew-\N u_0\right)^3 \, dx  \\
&\leq&  \int_{\{x: u_0 \leq -M_\delta\}} \{F(\unew) - F(u_0)\} + p + {\frac{(M+1)^2}{4}\max\{ 1, \frac{1}{\gamma^2} \}} \, p^2  \} \, dx.
\end{eqnarray*}
By choosing $\delta$ smaller if necessary, we can ensure that  {$\frac{(M+1)^2}{4}\max \{1, 1/\gamma^2 \} \, p \leq 1/2 \gamma$.}
The convexity of $F(\xi)$ for $\xi \leq 0$ then implies that
\begin{eqnarray*}
J(\unew) - J(u_0) < \int_{\{x: u_0 \leq -M_\delta \}} p\left \{F^{\prime}(\unew) + 1 + \frac{1}{2\gamma} \right \} \, dx. 
\end{eqnarray*}
As $M_\delta \geq M$, it is immediate from the definition of $M$ that $F'(\unew) = -f(\unew) = -f(-M_\delta) \leq -1  -\frac{1}{\gamma}$ on the set $\{x: u_0 \leq -M_\delta \}$. Therefore $J(\unew) - J(u_0) < 0.$ This contradicts the fact that $u_0$ is a minimizer in $\A$. 
\end{proof}

By Lemma~\ref{M1_const}, the minimizer $u_0$ is greater than $-(M + 1)$. Away from where $u_0$ equals $0, \beta$ or $1$, we can perturb $u_0$ by $C_0^{\infty}$ functions with small support to ensure that the perturbed function still lies inside $\A$. Setting $v_0 = \N u_0$, we can conclude after regularity bootstrap that $v_0 \in C^3[0,\infty)$. Moreover
$u_0 \in C^2$ and satisfies (\ref{FN_nonlinear_steady}a) except where $u_0$ equals $0, \beta$ or $1$. 

\section{Corner lemma} \label{sec_corner}
\setcounter{equation}{0}
\renewcommand{\theequation}{\thesection.\arabic{equation}}

To establish that $(u_0, \N u_0)$ is a standing pulse solution of \eqref{FN_nonlinear_steady}, we need to eliminate the possibility of an interval where $u_0$ equals $0, \beta$ or $1$. This requires a better understanding of the qualitative properties of $u_0$. In this section, we investigate the derivatives of the minimizer $u_0$ of $J$. From now on, $u_0$ always stands for the minimizer of $J$ and $v_0 = \N u_0$.

\begin{lem} \label{regularity}
Let $x_0 > 0$ and $\ell \in (0, x_0).$ If $u_0(x) \notin \{0, \beta, 1\}$ for $x \in [x_0-\ell, x_0)$ and $u_0(x_0) \in \{0, \beta, 1\}$, then both $\lim_{x \to x_0^{-}} u'_0(x)$ and $\lim_{x \to x_0^{-}} u''_0(x)$ exist. Moreover $u_0$ can be extended to a $C^\infty[x_0-\ell, x_0]$ function, satisfying (\ref{FN_nonlinear_steady}a) on $[x_0-\ell, x_0]$. A similar statement holds on the interval $(x_0, x_0+\ell]$.
\end{lem}

\begin{proof}
We only consider $u_0(x_0) = \beta$ and $u_0(x) \neq \beta$ on $[x_0 - \ell, x_0)$; the proof for the other cases are not different. Observe that $u_0 \in C^2[x_0 - \ell, x_0) \cap C[x_0-\ell, x_0]$ and $(u_0, v_0)$ satisfies (\ref{FN_nonlinear_steady}a) on $[x_0 - \ell, x_0)$. It is clear from (\ref{FN_nonlinear_steady}a) that $|u_0^{\prime\prime}|$ is 
bounded on $[x_0-\ell, x_0)$, and  $\lim_{x \to x_0^{-}} u''_0(x) = \frac{1}{d} \lim_{x \to x_0^{-}} (v_0(x)-f(u_0(x))$ exists. 
In view of 
\begin{equation*}
\lim_{x \to x_0^{-}} u'_0(x) =  u'_0(x_0 - \ell) + \lim_{x \to x_0^{-}} \int_{x_0 - \ell}^{x} u''_0(t) \, dt, 
\end{equation*}
the boundedness of the integrand guarantees that the limit exists. Hence $u_0 \in C^2[x_0 - \ell, x_0]$ and satisfies (\ref{FN_nonlinear_steady}a) on $[x_0 - \ell, x_0]$. 
Using typical regularity bootstrap by differentiating (\ref{FN_nonlinear_steady}a), we conclude that $u_0 \in C^{\infty}[x_0 - \ell, x_0]$.  
\end{proof}

Following a similar idea in \cite{Chen:2012}, the next lemma excludes the possibility of a sharp corner in the profile of $u_0$. 

\begin{lem} \label{corner_lemma}
Suppose $x_0$ and $\ell$ are positive numbers such that $u_0(x_0) \in \{0, \beta, 1\}$ and $u_0 \in C^1[x_0 - \ell, x_0] \cap C^1[x_0, x_0 + \ell]$. Then $\lim_{x \to x_0^{-}} u'_0(x) = \lim_{x \to x_0^{+}} u'_0(x)$.
\end{lem}

\begin{proof}
We first prove the case $u_0(x_0) = 0$. Suppose $\lim_{x \ra x_0^-} u^{\prime}_0 = a_1$ and $\lim_{x \ra x_0^+} u^{\prime}_0 = a_2$ with $a_1 \neq a_2$. By taking a sufficiently small $\ell_1 \leq \ell$, we may assume $u^{\prime}_0 = a_1 + o(1)$ on $[x_0-\ell_1, x_0]$ and $u^{\prime}_0 = a_2  + o(1)$ on $[x_0, x_0 + \ell_1]$. Let $y = L_1(x)$ be a straight line joining $(x_0 - \ell_1, u_0(x_0 - \ell_1))$ and $(x_0 + \ell_1, u_0(x_0 + \ell_1))$, whose slope is then given by $(a_1 + a_2)/2 + o(1)$. We obtain $\unew$ by trimming the corner of $u_0$ as follows:
\begin{equation*}
\unew =
\begin{cases}
u_0(x), \,\, \text{ if } x \leq x_0 - \ell_1, \\
L_1(x), \,\, \text { if } x_0 - \ell_1 \leq x \leq x_0 + \ell_1, \\
u_0(x), \,\, \text{ if } x \geq x_0 + \ell_1. 
\end{cases}
\end{equation*}
As this is a small perturbation from $u_0$, $\unew \in \A$. We will show that $J(\unew) < J(u_0)$. The change in the gradient term decreases, because 
\begin{eqnarray*}
\frac{d}{2} \int_{x_0 - \ell_1}^{x_0 + \ell_1} ( u_{new}^{\prime2} - u_0^{\prime2} ) \, dx &=& \frac{d}{2} \left \{ \int_{x_0 - \ell_1}^{0} (u_{new}^{\prime2} - u_0^{\prime2}) \, dx +  \int_{0}^{x_0 + \ell_1} (u_{new}^{\prime2} - u_0^{\prime2}) \, dx \right \} \\
&=& \frac{d\ell_1}{2}  \left \{ 2\left ( \frac{(a_1 + a_2)}{2} + o(1) \right )^2 - (a_1 + o(1))^2 - (a_2 + o(1))^2 \right \} \\
&=& \frac{d\ell_1}{2} \left \{ \frac{(a_1 + a_2)^2}{2}  - a_1^2 - a_2^2 + o(1) \right \} \\
&=& -\frac{d\ell_1}{4}\left\{ (a_1 - a_2)^2 + o(1) \right\} \\
&<& 0.
\end{eqnarray*}
By the mean value theorem, 
\[  \int_{x_0 - \ell_1}^{x_0 + \ell_1} \{ F(\unew) - F(u_0) \} \, dx = - \int_{x_0 - \ell_1}^{x_0 + \ell_1} f(\tilde{u})(\unew - u_0) \, dx  \]
for some $\tilde{u}$ between $u_0$ and $\unew$. With $\max_{-M \leq \xi \leq 1} | f(\xi) |$ being bounded and $|\unew - u_0| = O(\ell_1)$,
\begin{equation*}
\left \lvert \int_{x_0 - \ell_1}^{x_0 + \ell_1} \{ F(\unew) - F(u_0) \} \, dx \right \rvert \leq \ell_1O(\ell_1).
\end{equation*}
The change in the nonlocal term can be calculated by applying Lemma~\ref{nonlocal_difference}. Since both $\|\N u_0\|_{L^{\infty}}$ and $\|\N\unew\|_{L^{\infty}}$ are bounded,
\begin{align*}
& \left \lvert \integral \left \{ \frac{1}{2}\left( \unew \N \unew - u_0\N u_0 \right) + \frac{1}{4} \left ( (\N\unew)^4 - (\N u_0)^4 \right )\right\} \, dx \right \rvert  \\
\leq &\, \left \lvert \frac{1}{2}\int_{x_0 - \ell_1}^{x_0 + \ell_1} (\unew - u_0)\left(\N\unew + \N u_0\right) \, dx \right \rvert + { \frac{(M +1)^2}{4} \max \{1, \frac{1}{\gamma^2} \}}\, \int_{x_0 - \ell_1}^{x_0 + \ell_1}(\unew - u_0)^2 \, dx \\
\leq & \, \ell_1O(\ell_1).
\end{align*}
Observe that the changes in the potential term and in the nonlocal term are both negligible compared to that in the gradient term. Then $J(\unew) < J(u_0)$ contradicts the fact that $u_0$ is a minimizer in $\A$. The same argument can be used to treat the other cases. 
\end{proof}

\begin{remark} \label{oscillation}
In what follows, Lemma~\ref{corner_lemma} is referred to as a corner lemma, which does not require $u_0$ satisfy (\ref{FN_nonlinear_steady}b) on either $[x_0-\ell, x_0]$ or $[x_0, x_0 + \ell]$. 
\end{remark}

Let us consider the case $u_0(x_0) = 1$ and $u_0 \in C^1[x_0, x_0 + \ell]$ for some $\ell > 0$.
By taking $\ell$ sufficiently small, there are three possibilities for the behavior of $u_0$ on the left side of a neighborhood of $x_0$: 
\begin{enumerate}[] \itemsep=0.1mm
\item (P1) $u_0 < 1$ on $[x_0 - \ell, x_0)$;
\item (P2) $u_0 = 1$ on $[x_0 - \ell, x_0]$;
\item (P3) There exist $a_1 < b_1 \leq a_2 < b_2 \leq a_3 < b_3 \ldots$ in the interval $[x_0-\ell, x_0]$ such that 
\begin{equation*}
\begin{cases}
u_0 \text{ satisfies (\ref{FN_nonlinear_steady}b) on intervals } (a_n, b_n), \quad n = 1, 2, \ldots, \\
u_0 = 1 \quad \text{ on } [x_0 - \ell, x_0] \setminus \cup_{n=1}^{\infty} (a_n, b_n),
\end{cases}
\end{equation*}
with both $a_n \ra {x_0}^{-}$ and $b_n \ra {x_0}^{-}$. 
\end{enumerate}
The case of $u_0(x_0) = 0$ can be studied similarly with corresponding cases referred to as (Q1), (Q2), and (Q3), respectively. We denote the cases for $u(x_0) = \beta$ by (R1), (R2), and (R3). Except in the case (P3), (Q3), or (R3), $u_0 \in C^{\infty}[x_0 - \ell, {x_0 + \ell]}$ follows from Lemma~\ref{regularity} and the corner lemma. Moreover for (P3), (Q3), or (R3), the next lemma states that $\lim_{x\ra x_0} u_0^{\prime}(x)$ exists. As a consequence, we conclude that $u_0 \in C^1[0, \infty)$.

\begin{lem} \label{limit_point} 
Assume that $d \leq d_0$. If $x_0$ is a limit point stated in (P3), (Q3), or (R3), then {$u'_0(x_0) =0$ and $v_0(x_0)= v'_0(x_0) = 0$.} 
\end{lem}

\begin{proof}
On the interval $[a_n, b_n] {\subseteq [x_0 - \ell, x_0)} $, where $u_0$ satisfies (\ref{FN_nonlinear_steady}b), there is a $s_n \in (a_n, b_n)$ such that $u'_0(s_n) = 0$. Since $\|-f(u_0) + v_0\|_{L^{\infty}(a_n, b_n)} \leq C_1$ for some constant $C_1$ not depending on $x_0$ or $n$, integrating (\ref{FN_nonlinear_steady}b) yields $d \lvert u'_0(x) \rvert \leq C_1 \left | \int_{s_n}^{x} \, dt \right |$, which implies $\lvert u'_0(x) \rvert \leq C_1(b_n - a_n)/d.$
For $|b_n - a_n| \ra 0$ as $n \ra \infty$, it follows that $\|u'_0\|_{L^{\infty}(a_n, b_n)} \ra 0.$ Then $u_0 \in C^1[x_0-\ell, x_0]$ if we set $u'_0(x^{-}_0) = 0$. 
{Suppose Cases (P1), (P2), (Q1), (Q2), (R1) or (R2) occurs on the interval $[x_0,x_0+\ell]$, we see that
$u_0 \in C^1[x_0, x_0 + \ell]$ so that  $u'_0(x_0) =0$ is immediate from the corner lemma.} 
{On the other hand if $u_0$ satisfies an analogous situation (P3), (Q3) or (R3) on the interval $[x_0,x_0+\ell]$, 
the same argument  as for $[x_0-\ell,x_0]$ shows $u_0 \in C^1[x_0, x_0 + \ell]$ with $u_0^{\prime}(x_0^+) = 0$.
Hence in all scenario irrespective of what cases we have on the right of $x_0$, 
we have $u_0^{\prime}(x_0) =0$.} 

Next let us prove that {$v_0(x_0) = 0$ and $v'_0(x_0) = 0$.} Consider (P3) first. Since $u_0 \leq 1$ everywhere, by (\ref{FN_nonlinear_steady}b)
\begin{equation} \label{lim_sn}
f(u_0(s_n)) - v_0(s_n) = -du_0^{\prime\prime}(s_n) \leq 0.
\end{equation}
As $s_n \ra x^-_0$, we see that $f(u_0(x_0)) - v_0(x_0) \leq 0.$ On the other hand, since $u_0 \in C^2[a_n, b_n]$ by Lemma~\ref{regularity} and  $u_0(b_n) = 1$ with $u'_0(b_n) = 0$, 
\begin{equation} \label{lim_bn}
f(u_0(b_n)) - v_0(b_n) = -du_0^{\prime\prime}(b_n) \geq 0. 
\end{equation}
In this case, taking $b_n \ra x^-_0$ gives $f(u_0(x_0)) - v_0(x_0) \geq 0$. Therefore $f(u_0(x_0)) - v_0(x_0) = 0$ from \eqref{lim_sn} and \eqref{lim_bn} which implies that $v_0(x_0) = f(1) = 0$. Suppose now that $v'_0(x_0) < 0$. This together with $u'_0(x_0)=0$ gives $(f(u_0) - v_0)' \bigr\rvert_{x=x_0} = -v'_0(x_0) > 0$. Since $f(u_0(x_0)) - v_0(x_0) = 0$, it follows that $f(u_0(x)) - v_0(x) < 0$ on $[x_0-\delta, x_0)$ for some $\delta > 0$. This is incompatible with \eqref{lim_bn}. Similarly $v'_0(x_0) > 0$ would contradict (\ref{lim_sn}). Therefore $v'_0(x_0) = 0$.

If $u_0(x_0) = 0$ or $u_0(x_0) = \beta$, the proof of ${v_0(x_0) =} v_0'(x_0)=0$ is slightly different since $u_0$ can cross $0$ or $\beta$ in $(a_1, x_0)$; nevertheless due to the fact that $u_0$ can cross $0$ or $\beta$ only once, by choosing $a_1$ sufficiently close to $x_0$, $u_0$ does not change sign on $[a_1, x_0]$ in the case of (Q3), and either $u_0 \geq \beta$ or $u_0 \leq \beta$ on $[a_1, x_0]$ in the case of (R3). Then the rest of the proof is similar as above. We omit the details. 
\end{proof}

\section{Positivity of $v_0$} \label{sec_positive}
\setcounter{equation}{0}
\renewcommand{\theequation}{\thesection.\arabic{equation}}

Another essential qualitative property of the minimizer $u_0$ is the positivity of $v_0$. When the sign of $v_0$ is known, the energy change in the nonlocal term associated with the modification of $u_0$ becomes easier to quantify. As a result, Lemma~\ref{nonlocal_difference} turns out to be more useful when we apply the truncation technique. We begin with two lemmas which show that $v_0$ is partially positive. Then, we follow the idea in \cite{Chen:2015} to study the linearization of \eqref{FN_nonlinear_steady} which provides information crucial for showing $v_0 > 0$ everywhere. 

\begin{lem} \label{v0_pos}
If $u_0 \geq 0$ on $[0, \infty)$ is non-trivial, then $v_0 > 0$ everywhere.
\end{lem}

\begin{proof}
If $u_0 \geq0$, then $v_0 \geq \L u_0 > 0$ follows from Lemma~\ref{comparison}.
\end{proof}

\begin{lem} \label{v0(0)_pos}
No matter whether $u_0$ changes sign or not, $v_0(0) > 0$.
\end{lem}

\begin{proof}
If $u_0$ stays non-negative, the assertion follows immediately from Lemma~\ref{v0_pos}. Therefore assume $u_0$ changes sign at $x = x_2$. Suppose $v_0(0) \leq 0$. We claim that $v_0(x) < 0$ for all $x \in (0, \infty)$. Let us prove our claim on $(0, x_2]$ first. Its proof is divided into two cases:

\textit{Case 1:} Assume $v_0(0) < 0$. Since ${v_0}^{\prime\prime} - (\gamma + v_0^2)v_0 = -u_0 \leq 0$ on $[0, x_2]$, $v_0$ cannot have an interior negative minimum by the maximum principle. Moreover, with the boundary condition {$v_0^{\prime}(0) = 0$}, the Hopf lemma implies that the minimum occurs at $x = x_2$ and $v'_0 < 0$ on $(0, x_2]$. Hence $v_0(x) < v_0(0) < 0$ on $(0, x_2]$. 

\textit{Case 2:} Assume $v_0(0) = 0$. Then $v_0^{\prime\prime}(0) = -u_0(0) < 0$, and the boundary condition $v'_0(0) = 0$ implies that $v'_0(x) < 0$ in a neighborhood of $0$. This leads to the same conclusion as in Case 1.\\

On the interval $(x_2, \infty)$, since ${v_0}^{\prime\prime} - (\gamma + v_0^2)v_0 = -u_0 \,\geq 0$, $v_0$ cannot attain a non-negative interior maximum. From the fact ${v_0}(x_2) < 0$ and $v_0 \ra 0$ as $x \ra \infty$, it follows that $v_0 < 0$ on $[x_2, \infty)$. This finishes the proof of our claim. Now let $x_0$ be a point where $u_0$ attains its global minimum. Then $u^{\prime\prime}_0(x_0) \geq 0$. Since $u_0$ is negative in a neighborhood of $x_0$, we have {$f(u_0(x_0)) > 0$}, which implies $v_0(x_0) = du_0''(x_0) + f(u_0(x_0)) > 0.$ This contradicts our claim that $v_0 < 0$ on $(0, \infty)$.
\end{proof}

Let $d_1 \equiv \min \{d_0, \frac{\beta^2}{4(1+\beta\gamma)}\}$. It what follows, it is assumed that $d \leq d_1$. Observe that the system \eqref{FN_nonlinear_steady} can be expressed as
\begin{equation} \label{matrixForm}
\begingroup
\setlength\arraycolsep{0.5pt}
\begin{pmatrix} u_0 \\ v_0 \end{pmatrix}^{\prime\prime} - \, A \begin{pmatrix} u_0 \\ v_0 \end{pmatrix} \, = \, \begin{pmatrix} -\frac{u_0^2}{d}(1+\beta-u_0) \\ v_0^3 \end{pmatrix},
\endgroup
\end{equation} 
where 
\[A= 
\begin{pmatrix} \frac{\beta}{d} & \frac{1}{d}  \\ -1& \gamma \end{pmatrix}. \]
We now document the eigenvalues, and corresponding left and right eigenvectors of $A$. Details can be found in \cite{Chen:2015}.

\begin{enumerate}[(a)]
\item
Eigenvalues $\lambda_1, \lambda_2$ of $A$ are real and positive. Moreover they satisfy
\begin{equation} \label{eigenvalue_range}
 0 < \lambda_1 < \frac{\beta}{2d} < \frac{1}{2} \left (\gamma + \frac{\beta}{d} \right ) <\lambda_2 < \frac{\beta}{d}.
\end{equation}

\item For the eigenvalue $\lambda_1$, it has a right eigenvector {$\mathbf{a}=(-1,d\alpha_2)^{T}$} and a left eigenvector ${\mathbf l}_1=(1,\alpha_2)^{T}$,
where $\alpha_2 := \beta/d -\lambda_1 >0$.
For the eigenvalue $\lambda_2$, it has a right eigenvector {$\mathbf{b}=(-\alpha_2,1)^{T}$} and a left eigenvector ${\mathbf l}_2=(1,\alpha_1)^{T}$,
where $\alpha_1 := 1/d\alpha_2 >0$. 

It can be checked that
\[
0 < \alpha_1<\lambda_1< \frac{\beta}{2d} < \alpha_2 < \lambda_2< \frac{\beta}{d} \;,
\]
and 
\begin{equation} \label{evector_dot}
\mathbf{l}_1 \cdot \mathbf{a} >0\;, \qquad \mathbf{l}_2 \cdot \mathbf{b} <0 \;.
\end{equation}

\item The asymptotic behavior of $(u_0, v_0)$ at large $x$ can be studied by linearizing \eqref{matrixForm} about $(u,v) = (0,0)$:
\begin{equation} \label{linearized}
\begingroup
\setlength\arraycolsep{0.5pt}
\begin{pmatrix} \tilde{u} \\ \tilde{v} \end{pmatrix}^{\prime\prime}  - \,  A \begin{pmatrix} \tilde{u} \\ \tilde{v} \end{pmatrix} = \bold{0}.
\endgroup
\end{equation}
For \eqref{linearized} all the solutions decaying to $(0, 0)$ as $x \ra \infty$ are of the form 
\begin{equation} \label{linearization}
\begingroup
\setlength\arraycolsep{0.5pt}
\begin{pmatrix} u_0 \\ v_0 \end{pmatrix} \sim
\begin{pmatrix} \tilde{u} \\ \tilde{v} \end{pmatrix}  \, =\, C_1e^{-\sqrt{\lambda_1}x}\mathbf{a} + C_2e^{-\sqrt{\lambda_2}x}\mathbf{b}.
\endgroup
\end{equation} 
\end{enumerate}
\vspace{5mm}
While the linearization of \eqref{matrixForm} is the same whether the additional nonlinearity $v_0^3$ on its right hand side is present or not,
this nonlinearity  has to be taken into account when studying the solution on the entire interval $[0,\infty)$.

\begin{lem} \label{psi2}
Let $\psi_1 = u_0 + \alpha_2v_0$ and $\psi_2=u_0 + \alpha_1v_0$. Then for $i=1, 2$, $\psi_i \geq 0$ everywhere.
\end{lem}

\begin{proof}
We give a proof for $i=2$. A similar argument works when $i=1$. \\
\textit{Step 1:}  Define $\psi_2 = u_0 + \alpha_1v_0 = \mathbf{l}_2 \cdot  \binom{u_0}{v_0}$. Then $\psi_2 \in C^1[0, \infty)$. Away from the intervals where $u_0$ is identically $0, \beta,$ or $1$, $u_0 \in C^{\infty}$ and $(u_0, v_0)$ satisfies (\ref{matrixForm}). Premultiplying (\ref{matrixForm}) by $\ltwo$ yields
\begin{equation*} 
{\psi_2}^{\prime\prime} - \lambda_2\psi_2 = -\frac{u_0^2(1+\beta-u_0)}{d} + \alpha_1 v_0^3.
\end{equation*}
Let us subtract $\frac{1}{\alpha^2_1}\psi_2^3$ from both sides to get
\begin{align*}
{\psi_2}^{\prime\prime} - \lambda_2\psi_2 - \frac{1}{\alpha_1^2}\psi^3_2 &= -\frac{u_0^2(1+\beta-u_0)}{d} + \alpha_1 v_0^3 - \frac{1}{\alpha_1^2}\psi^3_2 \\
&= -\frac{u_0^2(1+\beta-u_0)}{d} + {\alpha_1 v_0^3} - \frac{1}{\alpha_1^2}(u_0^3 + 3\alpha_1u_0v_0\psi_2  + {\alpha_1^3v_0^3}) \\
&= -\frac{u_0^2}{d}(1+\beta-u_0) - \frac{u_0^3}{\alpha_1^2} - \frac{3u_0v_0}{\alpha_1}\psi_2, 
\end{align*}
which is equivalent to 
\begin{equation}\label{psi_equality}
{\psi_2}^{\prime\prime} - (\lambda_2 + \frac{1}{\alpha_1^2}\psi_2^2 - \frac{3u_0v_0}{\alpha_1})\psi_2 = -\frac{u_0^2}{d}(1+\beta - {(1-\frac{d}{\alpha_1^2})}u_0).
\end{equation}
Since $d \leq \frac{\beta^2}{4(1+\beta\gamma)} < \frac{\beta^2}{4}$ and  $\frac{1}{\alpha_1} = \beta - d\lambda_1  < \beta$, we have $\frac{d}{\alpha_1^2} < \frac{\beta^4}{4} < \, 1$. Hence, {$0 < 1-\frac{d}{\alpha_1^2} < 1$}. Together with $u_0 \leq 1$, it is clear that the sign of the right hand side of \eqref{psi_equality} is non-positive. With  $h(x) = \lambda_2 + \frac{1}{\alpha_1^2}\psi_2^2 - \frac{3u_0v_0}{\alpha_1},$
\begin{equation}
{\psi_2}^{\prime\prime} - h(x)\psi_2 \leq 0 \label{psi_max}.
\end{equation}
We remark that
\begin{align*}
h(x) &= \lambda_2 + \frac{1}{\alpha_1^2}(\psi_2^2 -3\alpha_1u_0v_0) \\
&=\lambda_2 + \frac{1}{\alpha_1^2}(u_0^2 + \alpha_1^2v_0^2 - \alpha_1 u_0v_0) \\
& > \, 0. 
\end{align*}
\noindent \textit{Step 2:} Suppose for contradiction $\psi_2 < 0$ somewhere. Define $b \equiv \sup\{x:\psi_2(x) < 0\},$ where $b=\infty$ is allowed. Since $\psi_2 \ra 0$ as $x \ra \infty$, it follows that $\psi_2(b) = 0$ or $\psi_2 \ra 0$ if $b=\infty$. In either case, there exists a $b_1 \in (0, \infty)$ such that $\psi_2(b_1) := -t_0 < 0$ and $\psi_2'(b_1) := t_1 > 0$. We claim $\psi_2'(x) > t_1$ on $(0, b_1]$. Let $a_1$ be a point in $(0, b_1)$. First we verify that $\psi_2(x)< -t_0$ and $\psi_2'(x) > t_1$ on $[a_1, b_1]$  under three possibilities: \vspace{2mm}

\textit{Case (A):} Suppose that $u_0 \neq 0$, $u_0 \neq \beta$, and $u_0 \neq 1$ on $[a_1, b_1]$. Since $\psi_2$ satisfies (\ref{psi_max}) on $[a_1, b_1]$, it cannot attain a non-positive minimum on $(a_1, b_1)$ by the maximum principle. Moreover, with $\psi'_2(b_1) > 0$, it follows from the Hopf lemma that $\psi_2' > 0$ on $[a_1, b_1]$. {Therefore $\psi_2(x) < \psi_2(b_1) = -t_0$ on $[a_1, b_1)$.} Putting this information back into (\ref{psi_max}), $\psi_2^{\prime\prime} (x) \leq h(x)\psi_2(x) < -h(x) t_0$ for all $x \in [a_1, b_1)$; consequently, $\psi_2'(x) \geq t_0 \int_x^{b_1} h(\xi) \, d \xi + t_1 > t_1$. 

\textit{Case (B):} Suppose $u_0 = 1$ on $[a_1, b_1]$. Since $u_0(b_1)=1$ and $u'_0(b_1) = 0$ by the corner lemma, it follows that $v_0(b_1) = -(t_0 + 1)/\alpha_1$ and $v'_0(b_1) = t_1/\alpha_1$. In view of 
\begin{equation} \label{caseB}
v_0^{\prime\prime} - (\gamma + v_0^2)v_0 = -1 < 0 \quad \text{on } [a_1, b_1],
\end{equation}
the maximum principle together with the Hopf lemma implies that $v_0(x) < v_0(b_1) = -(t_0 + 1)/\alpha_1$ on $[a_1, b_1)$. Then $v_0^{\prime\prime} < -\{ 1 + (\gamma + v_0^2)(t_0 + 1)/\alpha_1\}$ and consequently $v'_0(x) > t_1/\alpha_1 +\int_x^{b_1}  \{1+ (\gamma + v_0^2(\xi))(t_0 + 1)/\alpha_1\}\, d\xi$ for all $x \in [a_1, b_1)$. Combining with $u'_0 = 0$ on $[a_1, b_1]$ yields $\psi_2'(x) > t_1 + \int_x^{b_1} \{\alpha_1 + (\gamma + v_0^2(\xi))(t_0 + 1)\} \, d\xi> t_1$ {and therefore $\psi_2(x)< -t_0$}.

\textit{Case (C):} If $u_0 = \beta$ or $u_0 = 0$ on $[a_1, b_1]$, replacing $1$ by $\beta$ and $0$, respectively, the calculation in Case (B) will do. \vspace{2mm}

To finish our claim that $\psi_2^{\prime}(x) > t_1$ for all $x \in (0, b_1]$, it suffices to show that $(0, b_1]$ is a finite combination of Cases (A)-(C). Suppose there is an accumulation point $x_0$ such that $x \downarrow x_0^+$ with one of Cases (A)-(C) occurs alternatively in adjacent subintervals of $(x_0, b_1)$ or possibly a combination of such distributions. From what we have shown $\psi_2^{\prime} \geq t_1$ on $(x_0, b_1)$, so $\psi'_2(x_0) \geq t_1 > 0$ follows from $\psi_2 \in C^1[0, \infty)$. However, $u'_0(x_0) = v'_0(x_0) = 0$ by Lemma~\ref{limit_point}, which implies that $\psi'_2(x_0) = 0$. This is a contradiction. The same is true when $x_0$ is a limit point from the left. Hence there is no accumulation point, and therefore $\psi^{\prime}_2 \geq t_1 > 0$ on $(0, b_1)$. On the other hand, with $v_0(0)>0$ from Lemma~\ref{v0(0)_pos} and $\psi_2(0)< -t_0$, we see that $u_0(0)<0$. This is a contradiction since it follows from $x_1 > 0$ proved in Lemma~\ref{min_sequence} that $u(0) > 0$. An alternative proof that does not require $u(0) > 0$ to be known is given in the following: the continuity of $\psi_2^{\prime}$ implies that $\psi_2^{\prime}(0) \geq t_1 > 0$. Since $u_0 (0) \not \in \{ 0, \beta, 1 \}$, $u_0$ is smooth and satisfies \eqref{FN_nonlinear_steady} in a neighborhood $[0,\delta_1)$ of $x=0$. By choosing a test function $\varphi$ with support on $[0,\delta_1)$, a duplication of the proof of Lemma~\ref{natural_bndry} shows that the natural boundary condition $u_0'(0)=0$ is satisfied. 
Coupled with $v_0'(0)=0$, it follows that $\psi_2^{\prime}(0) = 0$, which is absurd. This completes the proof of $\psi_2 \geq 0$.
\end{proof}

\begin{lem} \label{pos_all}
If $d \leq d_1$, then $v_0 > 0$ everywhere. Moreover, if $u_0 \geq 0$ on $[0, x_2]$ and $u_0 \leq 0$
on $[x_2, \infty)$ for some $x_2$, then $v_0^{\prime} < 0$ on $[x_2, \infty)$ and $v_0 \downarrow 0$ as $x \ra \infty$. Once $u_0$ turns negative, then $u_0 < 0$ for all $x \in (x_2, \infty)$. 
\end{lem}

\begin{remark}
{The fact that $u_0$ changes sign at some $x_2 < \infty$ will be shown later in Lemma~\ref{sign_change}. The qualitative properties of $(u_0, v_0)$ stated in Lemma~\ref{pos_all} will therefore always hold.}
\end{remark}

\begin{proof}
It suffices to consider the case when $u_0$ changes the sign at $x=x_2$, otherwise Lemma~\ref{v0_pos} implies the positivity of $v_0$. Let us first consider the interval $[0, x_2]$, where
\[ v_0^{\prime\prime} - (\gamma + v_0^2)v_0 = -u_0 \leq 0.\]
By the maximum principle, $v_0$ cannot attain an interior non-positive minimum. Since $v_0(0) > 0$ and $v_0(x_2) = \psi(x_2)/\alpha_1 \geq 0$, it follows $v_0 > 0$ on $[0, x_2)$. We claim that $v_0(x_2) > 0$. For if not, the Hopf lemma implies $v_0^{\prime}(x_2) < 0$, and thus $v_0 < 0$ on $(x_2, x_2+\ep)$ for some $\ep > 0$. However from a different perspective,
\begin{equation} \label{v0_nonneg}
v_0 = \frac{1}{\alpha_1} (\psi_2 - u_0) \geq 0 \quad \text{ on $[x_2, \infty)$}
\end{equation}
by using Lemma~\ref{psi2}.
Therefore $v_0 > 0 $ on $[0, x_2]$. 

Next consider the interval $[x_2, \infty)$. The maximum principle applied to 
\begin{equation} \label{v0_positive}
v_0^{\prime\prime} - (\gamma + v_0^2)v_0 = -u_0 \geq 0
\end{equation}
implies $v_0$ cannot have an interior non-negative maximum. If $v_0$ touches $0$, it cannot go back up since $v_0$ then attains a positive maximum as it decays to $0$. Thus $v_0$ has to satisfy one of the following cases: \\

\noindent (A) $v_0$ decreases to 0 on $[x_2, \infty)$ with $v_0^{\prime} < 0$ by the Hopf lemma, or \\
(B) $v_0(z_0) = 0$ for some $z_0 > x_2$, where $z_0$ is the first point at which $v_0$ touches $0$.  \\

\noindent To eliminate (B), we apply the Hopf lemma to \eqref{v0_positive} on $[z_0, \infty)$ and conclude that $v_0^{\prime}(z_0) < 0$. This gives a rise to a contradiction since $v_0 \geq 0$ on $[x_2, \infty)$ as seen in \eqref{v0_nonneg}. 

The last statement is a consequence of the maximum principle applied to $du_0^{\prime\prime} - h(u_0)u_0 \geq 0$ on $(x_2, \infty)$ with $h(u) = (1-u)(\beta-u) \geq 0$ on the interval. 
\end{proof}

\noindent {The next corollary is an immediate consequence of the positivity of $v_0$ and Lemma~\ref{limit_point}.}

\begin{cor}
{The cases (P3), (Q3) and (R3) cannot occur.} 
\end{cor}
From Lemma~\ref{min_existence} we have $x_1>0$.  The above Corollary then implies either Case (P1) or (P2)  happens near $x=0$.
 In the former case, solution $u_0$ will be smooth near $x=0$ so that 
the natural boundary condition $u_0'(0)=0$ holds. For the latter case  when $u_0=1$ in a neighborhood of $x=0$, it is clear that $u_0'(0)=0$.  Thus we can conclude the followings.
\begin{cor} \label{bc0}
The minimizer $u_0$ satisfies $u_0'(0)=0$.
\end{cor}

\noindent With the new information $v_0 > 0$ everywhere, we in fact exclude the possibility that $u_0 =1$ on any interval. 

\begin{lem}
If $ d \leq d_1$ then $\beta_1 < \max u_0 < 1$.
\end{lem}

\begin{proof}
To show $\max u_0 < 1$, suppose there exists an $x_0 \in [0, x_1)$ such that $u_0(x_0) = 1$. 
Without loss of generality we can assume $u_0<1$ on $(x_0, x_0+\delta_1]$ for some small $\delta_1>0$.
Consequently $u_0$ satisfies (\ref{FN_nonlinear_steady}a) on $[x_0,x_0+\delta_1]$.
By making $\delta_1$ smaller if needed, we 
set $h(x) := u_0(x)(u_0(x) - \beta) >0$ on $[x_0,x_0+\delta_1]$.  This gives
\[d(u_0 - 1)^{\prime\prime} - h(x)(u_0 -1) = v_0 \geq 0. \]
By the maximum principle, $u_0 -1$ cannot attain an interior non-negative maximum on $(x_0,x_0+\delta_1)$ and moreover, the Hopf lemma dictates that 
$u_0'(x_0)<0$. Thus $x_0 \ne 0$ by Corollary~\ref{bc0}. With $x_0>0$, we see that $u_0>1$ in some interval $[x_0-\delta_2, x_0)$ for some small $\delta_2>0$.
This is a contradiction. 

Lastly, we need $\integral F(u_0) \,dx < 0$ for $J(u_0) < 0$ since all other terms are positive. It must hold that $\max u_0 > \beta_1$. 
\end{proof}

\section{On the constraint $u_0 = 0$} \label{sec_constraint0}
\setcounter{equation}{0}
\renewcommand{\theequation}{\thesection.\arabic{equation}}

At the moment we have not eliminated the possibilities of Cases (Q2) and (R2), i.e. there may exist intervals on which $u_0=\beta$ or
$u_0=0$. As a consequence 
 $x_1$ and $x_2$ as defined in the admissible set $\A$ may not be unique.
Let $\{ x_2 \}$ denote the set of points that represent any $x_2$.
With the established qualitative properties of $(u_0, v_0)$, we are ready to show that there are no intervals on which $u_0$ is identical to 0; to
be more precise the set $\{x_2\}$ has only 1 point at which $u_0$ changes sign. The truncation argument will serve as the key tool for the proofs in this section. 

\begin{lem} \label{nonloc_decrease}
Suppose we trim $u_0$ on a compact support such that $\unew \in \A$ with $\unew \leq u_0$. If $\|\unew - u_0\|_{L^{\infty}(0, \infty)} = O(\ep)$ for some $\ep > 0$, then for $\ep$ sufficiently small, the nonlocal energy decreases as well. That is,
\begin{equation*}
\integral \left \{ \left ( \frac{1}{2} \unew \N\unew + \frac{1}{4}(\N\unew)^4 \right) - \left ( \frac{1}{2}u_0\N u_0 + \frac{1}{4}  (\N u_0)^4 \right) \right \} \, dx < 0\;.
\end{equation*}
\end{lem}

\begin{proof}
With $u_{new} -u_0 \leq 0$, Lemma~\ref{N_increasing} gives 
${\cal N} u_{new} \leq {\cal N} u_0=v_0$. Let $I = \{x: \unew - u_0 \neq 0\}$ have a compact support.  If $\ep$ is sufficiently small, by continuity ${\cal N} u_{new}$ will remain positive on $I$. Then it follows from Lemma~\ref{nonlocal_difference} and Lemma~\ref{v_upperbnd} that
\begin{align*}
&\integral \left \{ \left ( \frac{1}{2} \unew \N\unew + \frac{1}{4}(\N\unew)^4 \right) - \left ( \frac{1}{2}u_0\N u_0 + \frac{1}{4}  (\N u_0)^4 \right) \right \} \, dx\\
= &\; \frac{1}{2}\int_{I} (\unew - u_0)(\N\unew + \N u_0) \, dx +  \frac{1}{4} \integral(\N\unew + \N u_0)(\N \unew - \N u_0)^3 \, dx \\
\leq& \; \frac{1}{2} \min_{x \in I} v_0(x) \int_{I} (\unew - u_0) \, dx +  {\frac{(M_1 + 1)}{2} \integral\lvert\N \unew - \N u_0\rvert^3 \, dx} \\
= & - \frac{1}{2}  \min_{x \in I}v_0(x)\|\unew - u_0\|_{L^1(0, \infty)} + \frac{(M_1 + 1)}{2} \|\N\unew - \N u_0\|^3_{L^3(0, \infty)} \\
\leq& - \frac{1}{2}  \min_{x \in I}v_0(x)\|\unew - u_0\|_{L^1(0, \infty)} + \frac{C_0(M_1 + 1)}{2} \|\N\unew - \N u_0\|^{3}_{H^1(0, \infty)}
\end{align*}
for some positive constant $C_0$, where the last inequality follows from the Sobolev embedding $H^1(0, \infty)\hookrightarrow  L^3(0, \infty)$. Finally by applying Lemma~\ref{difference}, we obtain
\begin{align}
&\integral \left \{ \left ( \frac{1}{2} \unew \N\unew + \frac{1}{4}(\N\unew)^4 \right) - \left ( \frac{1}{2}u_0\N u_0 + \frac{1}{4}  (\N u_0)^4 \right) \right \} \, dx \notag \\
\leq& -\frac{1}{2} \min_{x \in I}v_0(x)\|\unew - u_0\|_{L^1(0, \infty)} + \frac{C_0(M_1 + 1)}{2} \max\{1, 1/\gamma^3\} \|\unew - u_0\|^{3}_{L^2(0, \infty)} \notag \\
\leq& - \frac{1}{2} \min_{x \in I}v_0(x)\|\unew - u_0\|_{L^1(0, \infty)} + {\sqrt{2} \, C_0(M_1 + 1)^{5/2}}
\max\{1, 1/\gamma^3\} \|\unew - u_0\|^{3/2}_{L^1(0, \infty)}\label{change_in_O}  \\
<& \, 0 \notag
\end{align}
{for sufficiently small $\ep$, as $\|\unew - u_0\|_{L^1(0, \infty)}$ can be made arbitrarily small.}\\
\end{proof}

When we refer to $d_1$ in  the following lemmas, we understand that  it depends on $\gamma$, i.e. $d_1=d_1(\gamma)$.

\begin{lem} \label{u0(x2)}
{Suppose $\gamma \leq \gamma_0$ and $d \leq d_1$. Take the largest $x_1$ so that $u_0 < \beta$ on some small neighborhood $(x_1, x_1 + \delta]$ and, if $\{x_2\}$ is non-empty, take the smallest $x_2$  such that $u_0 > 0$ on some neighborhood $[x_2- \delta, x_2)$. Then $u'_0 < 0$ on the interval $[x_1, x_2)$; the same is true on $[x_1, \infty)$ if $\{x_2\}$ is empty. Moreover if $u_0$ changes sign at $x_2$, then $u'_0(x_2) < 0$ and $u_0 < 0$ on $(x_2, \infty)$.} 
\end{lem}

\begin{remark} 
We will establish in Lemma~\ref{u0neq0} and Lemma~\ref{sign_change} that
 $u_0$ changes sign at a unique $x_2 < \infty$; therefore $u_0$ will satisfy all the qualitative properties stated in Lemma~\ref{u0(x2)}.
\end{remark}

\begin{proof}
Suppose \{$x_2$\} is nonempty and there exist $x_1 < y_1 < y_2 < x_2$ such that $0 < u_0(y_1) < u_0(y_2)$. Since $u_0(x_2) = 0$, a local maximum of $u_0$ is attained  between $y_1$ and $x_2$, thereby creating a hump. The top of the hump can even go up all the way and form an interval on which $u_0=\beta$. 
Take a small positive $\ep$ and let
\begin{equation*}
\unew(x) =
\begin{cases}
u_0(x)  -\ep , \, & \text{   if } x \geq y_1 \; \mbox{and} \; u_0(x) \geq {\max_{[y_1,x_2]}} u_0 - \ep,  \\
u_0(x) ,  \,  & \text{   otherwise. }
\end{cases}
\end{equation*}
In other words we trim a small height $\ep$ from the top of the hump and obtain a $u_{new} \in \A$. Upon trimming,
it is clear that the gradient energy decreases. As {$F(\xi)$ is strictly monotone increasing 
for $\xi \in [0, \beta]$}, the potential energy also decreases. 
Finally since $\unew - u_0$ has a compact support, the nonlocal energy decreases as well by Lemma~\ref{nonloc_decrease}.
These lead to $J(u_{new})<J(u_0)$, contradicting $u_0$ being a minimizer. It forces us to conclude that 
$u_0$ is non-increasing on $[x_1, x_2]$. By Lemma~\ref{regularity}, $u_0 \in C^{\infty}[x_1, x_2]$ and satisfies (\ref{FN_nonlinear_steady}a) on the interval. Since $du_0^{\prime\prime} = v_0 - f(u_0) > 0$ on $[x_1, x_2]$, the Hopf lemma implies that $u_0^{\prime} < 0$ on $[x_1, x_2)$. If $\{x_2\}$ is empty,  the same argument still works if we take $x_1 < y_1 < y_2 < \infty$. This eads to $u_0'<0$ on $[x_1,\infty)$ in this case.

For finite $x_2$, only one of the followings will happen: \\ 
(a) $u_0$ becomes negative on $(x_2, x_2 + \delta]$ for some finite $\delta>0$, \\
(b) {$u_0 = 0$ on $[x_2, x_2 + \delta]$ and $u_0 < 0$ on $(x_2 + \delta, x_2 + \delta + \delta_1]$ for some positive $\delta$ and $\delta_1$, or} \\
(c) $u_0 = 0$ on $[x_2,  \infty)$.
 \vspace{2mm}

Assume $u_0$ changes sign at $x_2$. Then we need to consider 
only cases (a) and (b).
Suppose case (b) occurs. Let $h(x) = -(u_0 - \beta)(1-u_0)$, which is positive on $[x_2 + \delta, \infty)$. Since $du_0^{\prime\prime} - h(x)u_0 = v_0 > 0$ on the interval, we apply the Hopf lemma to conclude that $u_0^{\prime}(x_2 + \delta) < 0$. This contradicts the result from the corner lemma that $u_0^{\prime}(x_2 + \delta) = 0.$ Therefore only case (a) holds and $u_0^{\prime}(x_2) < 0$ follows from the Hopf lemma on $[x_2, x_2 + \delta]$. The same Hopf lemma argument will also prevent $u_0$ from touching zero again on $(x_2,\infty)$, which leads us to conclude that $u_0 < 0$ on $(x_2, \infty)$.
\end{proof}
\noindent Next we eliminate case (c) in the above proof.

\begin{lem} \label{u0neq0}
Suppose $\gamma \leq \gamma_0$ and $d \leq d_1$. Whether $u_0$ changes sign or not, there cannot be an interval $[a, b]$ where $u_0 = 0$. 
\end{lem}

\begin{proof}
In view of Lemma~\ref{u0(x2)}, it suffices to eliminate case (c) in its proof. 
Let $u_0 = 0$ on a finite interval $[a, b] \subset [x_2, \infty)$. Take $\ep > 0$ small and define 
\begin{equation*}
\unew(x) =
\begin{cases}
 -\ep \sin\left(\frac{\pi(x-a)}{b-a} \right), \, & \text{   if } x \in [a, b], \\
u_0,  \,  & \text{ otherwise. }
\end{cases}
\end{equation*}
It is clear that $\unew \in \A$. 
Let us set $v_0 = \N u_0$ and $\vnew = \N\unew$. {With $\| \unew - u_0\|_{L^1(0, \infty)} = \frac{2(b-a)}{\pi}\ep$, it follows from the calculation in \eqref{change_in_O} that the change in nonlocal energy is given by}
\begin{align*}
&\integral \left \{ \left ( \frac{1}{2} \unew \vnew + \frac{1}{4}\vnew^4 \right) - \left ( \frac{1}{2}u_0v_0 + \frac{1}{4}  {v_0}^4 \right) \right \} \, dx \\
&\leq  - \min_{x \in I}v_0(x)\frac{(b-a)}{\pi}\ep + {\sqrt{2} \, C_0(M_1 + 1)^{5/2}} \max\{1, 1/\gamma^3\} \left(\frac{2(b-a)}{\pi}\ep\right)^{3/2}.
\end{align*}

{Since $0 \leq F(\xi) \leq (1+\beta)\xi^2/2$ for small $\xi$, together with $u'_0 = 0$ and $F(u_0) = 0$ on $[a, b]$, we obtain}
\begin{align*}
J(\unew) - J(u_0) =& \int_{a}^{b} \left \{ \frac{d}{2}\unew^{\prime2} + F(\unew) \right \}  \, dx  \\
&\quad + \integral \left \{ \left ( \frac{1}{2} \unew \N\unew + \frac{1}{4}(\N\unew)^4 \right) - \left ( \frac{1}{2}u_0\N u_0 + \frac{1}{4}  (\N u_0)^4 \right) \right \} \, dx \\
=& \; O(\ep^2) - C \ep + O(\ep^{3/2}) 
\end{align*}
{for some positive constant C. Then $J(\unew) < J(u_0)$ if we choose $\ep$ sufficiently small, but this contradicts the fact that $u_0$ is a minimizer.} 
\end{proof}

\begin{lem} \label{sign_change}
If $\gamma \leq \gamma_0$ and $d \leq d_1$, then
\begin{enumerate}[(i)] \itemsep=0.25mm
\item {$u_0$ has a slow decay at $+\infty$;}
\item $u_0$ changes sign. 
\end{enumerate}
\end{lem}

\begin{proof}
By Lemmas \ref{u0(x2)} and \ref{u0neq0}, there exists some large $y_1 > 0$ such that $u_0$ vanishes to $0$ and $u_0 \neq 0$ on $[y_1, \infty)$. Therefore, we can study the behavior of $(u_0, v_0)$ near $+\infty$ from the linearization in \eqref{linearization}. If $u_0$ has a fast decay at $+\infty$, then $\begin{pmatrix} u_0 \\ v_0 \end{pmatrix} \sim C_2e^{-\sqrt{\lambda_2}x}\mathbf{b}$ with $C_2 \neq 0$. Therefore $\psi_2 = \mathbf{l_2}\cdot\begin{pmatrix} u_0 \\ v_0 \end{pmatrix} \sim C_2e^{-\sqrt{\lambda_2}x}\mathbf{l_2}\cdot\mathbf{b}$. {Since $\psi_2 \geq 0$ by Lemma~\ref{psi2} and $\mathbf{l}_2 \cdot \mathbf{b} <0$ by \eqref{evector_dot}, it follows that $C_2 < 0$.} Recall that $\mathbf{b} = (-\alpha_2, 1)^{T}$, where $\alpha_2 = \beta/d - \lambda_1 > 0$. Then $u_0 > 0$ and $v_0 < 0$ at large $x$, which contradicts the positivity of $v_0$. The proof of (a) is now complete.

With known slow decay, $\begin{pmatrix} u_0 \\ v_0 \end{pmatrix} \sim C_1e^{-\sqrt{\lambda_1}x}\bold{a}$ with $C_1 \neq 0$. Taking inner product with $\mathbf{l_1}$ yields $\psi_1 \sim C_1e^{-\sqrt{\lambda_1}x}\mathbf{l_1}\cdot\bold{a}$. It follows again from \eqref{evector_dot} and the positivity of $v_0$ that $C_1 > 0$. With $\mathbf{a}=(-1,d\alpha_2)^{T}$, it is clear that $u_0$ is negative at large $x$. Therefore, $u_0$ must change sign at some finite $x_2$.
\end{proof}

\noindent {The above lemma eliminates case (c) in the proof of Lemma~\ref{u0(x2)}. As a consequence we have the following corollary.}

\begin{cor} \label{x2_unique}
{Suppose $\gamma \leq \gamma_0$ and $d \leq d_1$. Let $x_1 = \inf\{y: u_0(x) < \beta$ if $x \in (y, \infty)$\}. Then the minimizer 
$u_0 \in C^{\infty}[x_1, \infty)$. In fact $u_0$ changes sign and satisfies (\ref{FN_nonlinear_steady}a) on this interval. Moreover the set $\{x_2\}$ contains a single point, i.e. $u_0$ crosses $0$ at only one point.}
\end{cor}

\section{On the constraint $u_0 = \beta$} \label{sec_constraintBeta}
\setcounter{equation}{0}
\renewcommand{\theequation}{\thesection.\arabic{equation}}

In this section, we establish that $(u_0, v_0)$ is the standing pulse solution of \eqref{FN_nonlinear_steady} by ruling out the possibility that $u_0$ equals to $\beta$ on any interval. We exploit the fact that \eqref{FN_nonlinear_steady} is a Hamiltonian system with
\begin{equation*}
\frac{1}{2}{v_0^{\prime2}} - \frac{\gamma}{2}v_0^2 - \frac{1}{4}v_0^4 + u_0v_0 - \frac{d}{2}u_0^{\prime2} + F(u_0) =  0
\end{equation*}
and that this identity is still valid on $(-\infty,\infty)$ even when $u_0 = \beta$ on an interval where \eqref{FN_nonlinear_steady} fails. Note that  
in the event of such an interval exists,  $u_0$ may not be $C^2$ at the boundary points of the interval.

\begin{lem}
Even if there are intervals where $u_0 = \beta$, $(u_0, v_0)$ satisfies 
\begin{equation} \label{Hamiltonian}
\frac{1}{2}{v_0^{\prime2}} - \frac{\gamma}{2}v_0^2 - \frac{1}{4}v_0^4 + u_0v_0 - \frac{d}{2}u_0^{\prime2} + F(u_0) = 0.
\end{equation}
\end{lem}
\begin{proof}
It can be seen from the linearization of $(u_0, v_0)$ that both $u_0$ and $v_0$ die down exponentially as $x \ra \infty$. Since $u_0^{\prime\prime}$ and
$v_0^{\prime\prime}$ are bounded, the standard interpolation theorem implies that $u_0^{\prime}$ and $v_0^{\prime}$ also die down exponentially. 

On the interval $[x_1, \infty)$, we multiply (\ref{FN_nonlinear_steady}a) by $-u_0^{\prime}$ and (\ref{FN_nonlinear_steady}b) by $v_0^{\prime}$, sum the resulting equations, and then integrate to obtain
\begin{equation}\label{ham_const}
\frac{1}{2}{v_0^{\prime2}} - \frac{\gamma}{2}v_0^2 - \frac{1}{4}v_0^4 + u_0v_0 - \frac{d}{2}u_0^{\prime2} + F(u_0) =  \mbox{constant}.
\end{equation}
By taking the limit as $x \ra \infty$, the integration constant is clearly zero as in (\ref{Hamiltonian}). This continues to be valid from the right until $u_0 = \beta$ on an 
interval $[a, b]$ with $b \leq x_1$ when (\ref{FN_nonlinear_steady}a) fails to hold. Note that $u_0^{\prime}(b^+) = u_0^{\prime}(b^-) = 0$ by the corner lemma. On $[a,b]$ where $u_0 = \beta$, it follows from (\ref{FN_nonlinear_steady}b) that $v_0^{\prime\prime} - \gamma v_0 - v_0^3 + \beta = 0$. Therefore
\begin{equation*}
\frac{d}{dx} \left ( \frac{1}{2} v_0^{\prime2} - \frac{\gamma}{2}v_0^2 - \frac{1}{4}v_0^4 + \beta v_0 \right ) = 0,
\end{equation*}
which implies that the left-hand side of (\ref{ham_const}) does not change on $[a, b]$. Since (\ref{Hamiltonian}) holds at $x = b$, this
constraint continues to be valid on $[a, b]$. Once $x < a$ with $u_0 > \beta$, both (\ref{FN_nonlinear_steady}a) and (\ref{FN_nonlinear_steady}b) are satisfied so that (\ref{ham_const}) holds. Then evaluating (\ref{ham_const}) at $x=a$ implies that the integration constant is zero. Hence, (\ref{Hamiltonian}) holds everywhere. 
\end{proof}

\begin{lem} \label{lem_gamma1}
 Let $\beta \in (1/3,1/2)$ (so that $\gamma_0>0$). 
Suppose $\gamma <  \gamma_1 \equiv \min\{\gamma_0, 2(\beta + F(\beta))- 1/2 \}$ and $d \leq d_1$. Then there cannot be an interval $[a, b] \subset [0, x_1]$ where $u_0 = \beta$. In fact there is no point at which $u_0 = \beta$ and $u_0^{\prime} = 0$, and the set $\{ x_1 \}$ has only a unique point.
\end{lem}

\begin{proof}
Assume there is an interval $[a, b] \subset [0, x_1]$ such that $u_0 = \beta$. By the corner lemma, $u_0^{\prime}(b) = 0$. Since $v_0 \leq 1$ as shown in Lemma~\ref{v_upperbnd} and $F(\beta) {= (2\beta^3 - \beta^4)/12} > 0$, on evaluating (\ref{Hamiltonian}) at $x=b$ we obtain
\begin{eqnarray*}
\frac{1}{2}(v_0^{\prime}(b))^2 &=& v_0(b)\left\{ \frac{\gamma}{2}v_0(b) + \frac{1}{4}v_0^3(b) - \beta \right\} - F(\beta) \\
& \leq & v_0(b)\left\{ \frac{\gamma}{2}v_0(b) + \frac{1}{4}v_0^3(b) - \beta - F(\beta)\right\} \\
&< & 0
\end{eqnarray*}
when $\frac{\gamma}{2}v_0(b) + \frac{1}{4}v_0^3(b) - \beta - F(\beta) \leq \frac{\gamma}{2} + \frac{1}{4} - \beta - F(\beta)<0$.
But $\frac{1}{2}(v_0^{\prime}(b))^2 < 0$ is absurd, and thus no such interval $[a, b]$ exists. It is clear from the proof that there is no point at which $u_0 = \beta$ and 
$u_0^{\prime} = 0$. As a consequence, the point $x_1$ is unique.
\end{proof}

At this point, we have completely removed the possibility that $u_0$ equals to one of the constraints imposed on $\A$. By the regularity estimates $u_0$ and $v_0$ are $C^{\infty}[0, \infty)$ functions. Extending them to be even functions on $(-\infty, \infty)$, we conclude that $(u_0, v_0)$ is a standing pulse solution to \eqref{FN_nonlinear_steady} satisfying $\lim_{|x| \ra \infty} (u_0, v_0) = (0, 0)$. This finishes the proof of Theorem \ref{theorem1}. 
We note that  a plot of  $\gamma_1$ in Lemma~\ref{lem_gamma1} versus $\beta$ has been represented in Figure~\ref{fig:1}. A better estimate
can make $\gamma_1$ larger.

Various qualitative properties of $u_0$ have already been investigated in the previous lemmas. To finish the proof of Theorem \ref{theorem2}, we show the following qualitative property of $u_0$.

\begin{lem}
Suppose $\gamma < \gamma_1$ and $d \leq d_1$. Then $u_0$ has a unique negative local minimum point on $[0, \infty)$ which is also the global minimum point.
\end{lem}

\begin{proof}
Recall from Lemma~\ref{pos_all} that $v_0'<0$ on $[x_2, \infty)$. Then $du_0^{\prime\prime\prime} +f'(u_0)u_0^{\prime} = v^{\prime}_0 < 0$ on $[x_2, \infty)$. Since $u_0 \leq 0$ on $[x_2, \infty)$, we have $f^{\prime}(u_0) \leq 0$ on that interval. 

Next from Lemma~\ref{sign_change} we know that $u_0 \sim -C_1e^{-\sqrt{\lambda_1}x}$ at $+\infty$ which implies that $u_0^{\prime} > 0$ at some large $x$. Since $u_0^{\prime}$ cannot attain a non-positive minimum on $[x_2, \infty)$ by the maximum principle, $u_0^{\prime}$ has to increase from a negative value at $x_2$ to $0$ at some $y_0 \in (x_2, \infty)$. Moreover $u_0'(y_0) > 0$ by using the Hopf lemma on $[y_0, \infty)$, and once $u_0'$ turns positive it cannot become negative again. Correspondingly $u_0$ decreases from $0$ at $x_2$ to a negative local minimum at $y_0$, and then increases to $0$ as $x \ra \infty$. Hence $u_0$ has a unique negative local minimum at $x = y_0$ on the interval $[x_2, \infty)$, which then is the global minimum of $u_0$ on the entire interval $[0, \infty)$.
\end{proof}


\end{document}